\documentclass[10pt,twoside,openany,UTF8,CJK]{article}
\usepackage{amssymb,amsthm}
\usepackage[tbtags]{amsmath}
\usepackage{cite}
\usepackage{indentfirst}
\usepackage{cases}
\usepackage{graphicx}
\usepackage{subfigure}
\usepackage{lineno}
\usepackage{enumerate}
\usepackage{threeparttable}
\usepackage{multirow}
\usepackage{booktabs}

\usepackage{bm}
\usepackage{CJK}
\usepackage{fancyhdr}
\usepackage{ifthen}
\usepackage{lipsum}
\allowdisplaybreaks[4]
\usepackage[misc]{ifsym}
\usepackage{caption}
\usepackage{float}
\usepackage{appendix}
\usepackage[format=hang]{caption}

\usepackage[colorlinks,
linkcolor=blue,
anchorcolor=blue,
citecolor=red]{hyperref}

\newtheorem{Definition}{Definition}[section]
\newtheorem{Lemma}{Lemma}[section]
\newtheorem{Theorem}{Theorem}[section]
\newtheorem{Remark}{Remark}[section]

\newtheorem{Assumption}{Assumption}[section]

\newcounter{saveeqn}

\makeatletter
\evensidemargin
\oddsidemargin
\makeatother

\title{\bf  A Pressure-Stabilized Continuous Data Assimilation Reduced Order Model}
\author{
	Xi Li,\footnote{School of Mathematics, Sichuan University, Chengdu, Sichuan 610064, China (li\_xi@stu.scu.edu.cn). The work of this author was supported by the National Natural Science Foundation of China(Grant No. 11971337).}
	\ Youcai Xu\footnote{School of Mathematics, Sichuan University, Chengdu, China(xyc@scu.edu.cn).}
	\ and Minfu Feng\footnote{Corresponding author. School of Mathematics, Sichuan University, Chengdu, Sichuan 610064, China (fmf@scu.edu.cn). The work of this author was supported by the National Natural Science Foundation of China(Grant No. 11971337).}
}
\date{}

\begin{document}
	\maketitle
	\newcommand\blfootnote[1]{%
		\begingroup
		\renewcommand\thefootnote{}\footnote{#1}%
		\addtocounter{footnote}{-1}%
		\par\setlength\parindent{2em}
		\endgroup
	}
	\captionsetup[figure]{labelfont={bf},labelformat={default},labelsep=period,name={Fig.}}
	\captionsetup[table]{labelfont={bf},labelformat={default},labelsep=period,name={Tab.}}
	
	\begin{abstract}
		We present a novel reduced-order pressure stabilization strategy based on continuous data assimilation(CDA) for two-dimensional incompressible Navier-Stokes equations. A feedback control term is incorporated into pressure-correction projection method to derive the Galerkin projection-based CDA proper orthogonal decomposition reduced order model(POD-ROM) that uses pressure modes as well as velocity’s simultaneously to compute the reduced-order solutions. The greatest advantage over this ROM is circumventing the standard discrete inf-sup condition for the mixed POD velocity-pressure spaces with the help of CDA which also guarantees the high accuracy of reduced-order solutions; moreover, the classical projection method decouples reduced-order velocity and pressure, which further enhances computational efficiency. Unconditional stability and convergence over POD modes(up to discretization error) are presented, and a benchmark test is performed to validate the theoretical results. \\
		
	\end{abstract}
	
	\baselineskip 15pt
	\parskip 10pt
	\setcounter{page}{1}
	
	\section{Introduction}
	\noindent Reduced Order Models(ROM) are surrogate models which are devised to replace full order models(FOM), e.g., finite element methods or finite volume methods, to deliver reasonably accurate results at a low computational cost. This is a case where the aim of time extrapolation is required to deal with long-time numerical simulations; or where multi-parameter scenarios occur when physical/geometric parameters are changing frequently. The ways to construct reduced-order spaces and to formulate projection-based ROM are generally including proper orthogonal decomposition(POD), greedy algorithm, etc., and Galerkin or Petrov-Galerkin projection. We herein adopt POD-ROM by Galerkin projection(G-POD-ROM), more details and extensions about ROM can be found in \cite{RBM-BOOK-1, RBM-BOOK-2}.  \\
	\indent ROM have been widely applied to many different fields; specifically in the field of fluids dominated by relatively few recurrent spatial structures, ROM have been successfully used as an efficient and effective numerical method to solve equations like  Euler\cite{Euler-POD-ROM-2021-JCP}, convection-diffusion-reaction\cite{SUPG-POD-ROM-2022-CMWA}, incompressible Navier-Stokes(NSE)\cite{Luo-POD-NS} and Rayleigh–Bénard convection\cite{RBC-POD-ROM-2019-AMM, RBC-POD-ROM-2020-ICHMT}, etc. In particular, when handling NSE in primitive variables with POD-ROM where mixed finite element method is used as direct numerical simulation(DNS), it is an open question that it is difficult to prove some certain compatibility condition between the data-based reduced-order mixed spaces, such as $LBB_r$ condition(a counterpart in POD-ROM of the mixed finite element inf-sup/$LBB_h$ condition), thus causing spurious numerical oscillations; while in some scenarios pressure variable is crucial and non-negligible\cite{POD-PressurePoissonEquation-JFM-2005}. To handle this, several techniques or methods have been proposed from different aspects, e.g. expanding the reduced-order velocity space \textit{a priori} based on the satisfaction of the $LBB_r$ condition by supremizer enrichment\cite{Supremizer-INFSUP-2013-NM, POD-Supremizer-2015-IJNME}, recovering reduced-order pressure \textit{a posteriori} by pressure Poisson equation(PPE)\cite{PPE-ROM-2009-TCFD, NS-POD-2014-JCP} or by pressure stabilized Petrov-Galerkin\cite{PSPG-POD-ROM-2022-CMAME}, adding pressure stabilization terms into G-POD-ROM\cite{SUPG-POD-ROM-2015-CMAME, VMS-POD-2019-ACM, Ali-PODInfSup-CMWA-2020, POD-LPS-SINUM-2021} or relaxing incompressible constraint by penalty methods/artificial compression methods/projection methods\cite{POD-AC-SINUM-2020, MyFirstPaper}. These methods have their strengths, but also some weaknesses; for example, using supremizer enrichment requires additional construction of exact or approximate supremizer space, and using PPE would introduce artificial boundary conditions, PSPG or other stabilization techniques will face choices of stabilization parameters; among those methods, the projection method, in the form of eliminating end-of-step velocity, is superior to others aforementioned on aspects that it decouples velocity-pressure variables and owns modular implementation with a fixed pressure-incremental parameter(i.e., $\Delta t$, the timestep size).  \\
	\indent Pressure-correction projection methods have become a classical method due to their decoupling, efficiency, and ease of implementation. A natural idea is: \textit{How to derive a pressure-correction projection ROM?} Data assimilation(DA) refers to a wide class of techniques for incorporating experimental/observational data into a simulation. We herein use a type of DA introduced in \cite{CDA-2014-JNS} and called continuous data assimilation(CDA) where a feedback control term is merged into the original model, and this term nudges simulated solutions towards experimental/observation data, thus improving the accuracy of numerical solutions or regularizing the original ill-posed physical model. The stability and convergence of CDA, at the DNS level, have been demonstrated in \cite{CDA-FEM-2019-CMAME, CDA-SINUM-2020} and Zerfas, Rebholz, Schneier, and Iliescu\cite{CDA-2019-CMAME} firstly incorporate this technique into ROM to enhance the accuracy of reduced-order solutions whose lack is also one of the main deficiencies of ROM and thus owns long-term attention in ROM community. Other similar works include \cite{CDA-POD-2022-JCAM} etc. We notice that there are not too many investigations about using CDA to help obtain high-accuracy reduced-order solutions, and the two successful combinations aforementioned both aim to reduced-order velocity by canceling pressure and then solving Burgers-type NSE ROM. Inspired by the statement in \cite{ROM-Clusures-2022} that in many important practical applications, the G-POD-ROM’s inaccuracy often manifests itself in the form of spurious numerical oscillations, we use CDA to enhance reduced-order pressure’s accuracy, and we find that, by doing so, the $LBB_r$ condition is circumvented and thus obtaining \textit{a priori} reduced-order pressure.  \\
	\indent In this manuscript, we propose, analyze and validate a CDA pressure-correction projection G-POD-ROM(CDA-Proj-POD-ROM). This correction version among projection methods allows us to obtain the desired highly accurate high-fidelity solution at a low computational complexity, thus saving the preparation time of the offline stage. The ROM, obtained by Galerkin projection, inherits the form of decoupling reduced-order velocity-pressure variables, thus avoiding the saddle-point system coming from solving the NSE numerical scheme. Velocity and pressure nudging terms are both incorporated into ROM, where the former is used to enhance the reduced-order velocity accuracy and thus keeping the reduced-order system energy stable; and the latter aims at avoiding the appearance of spurious numerical oscillations in a manner of improving reduced-order pressure accuracy, which finally allows obtaining high-accuracy reduced-order pressure simultaneously with velocity. To the best of the authors' knowledge, this is the first time that CDA is found to be served as a pressure stabilization technique in G-POD-ROM. The main contribution of this manuscript is to provide a novel, alternative way of obtaining stable, high-accuracy reduced-order pressure.  \\
	\indent The rest of the manuscript is organized as follows: In section \ref{section-2}, we introduce some preliminaries and notations necessary for our scheme and analysis. In section \ref{section-3}, we propose our CDA-Proj-POD-ROM and provide numerical analysis. In section \ref{section-4}, we perform a numerical experiment to validate our ROM. Finally, in section \ref{section-5}, we draw conclusions and outline our future directions.  \\ 
	\indent Throughout this paper, we use $C$ to denote a positive constant independent of $\Delta t$, $h$, not necessarily the same at each occurrence. Especially, $C(f)$ means this constant are relevant to the functions or parameter $f$.
	
	\section{Preliminaries and Notations}\label{section-2}
	\noindent Let $\Omega \subset \mathbb{R}^2$ is an open bounded domain with a sufficiently smooth boundary $\partial \Omega$. $L^p(\Omega),\, 1\leq p\leq \infty$, and its inner product $(\cdot,\cdot)$ and norm $\|\cdot\|$ are defined in usual ways. The Sobolev space $W^{m,p}(\Omega)$ and its norm $\|\cdot\|_{m,p}$ are also standard in the sense of \cite[p.23]{Brenner2008}. Furthermore, we use the abbreviation $H^m(\Omega) := W^{m,2}(\Omega)$ and $\|\cdot\|_m := \|\cdot\|_{m,2}$, it’s a Hilbert space with a scalar product. We denote by $H_0^1(\Omega)$ the space of functions of $H^1(\Omega)$ with vanishing trace on $\partial \Omega$ and by $H^{-1}(\Omega)$ its dual space. Vector analogs of the Sobolev spaces along with vector-valued functions are denoted by boldface letters, for instance, $\boldsymbol{H}^m(\Omega) := (H^m(\Omega))^2$.  \\
	\indent We consider the nonstationary incompressible Navier-Stokes equations(NSEs) in primitive variables with no-slip boundary conditions: 
	\begin{equation}\label{NSE}
		\begin{aligned}
			&\partial_t \boldsymbol{u} + \boldsymbol{u}\cdot\nabla\boldsymbol{u} - \nu\Delta \boldsymbol{u} + \nabla p = \boldsymbol{f},\quad\text{and}\quad\nabla\cdot \boldsymbol{u} = 0,\quad \text{in}\;\Omega\times(0,T],\\
			&\boldsymbol{u} = \boldsymbol{0},\;\text{on}\;\partial\Omega,\quad\boldsymbol{u}(0,x) = \boldsymbol{u}_{0}(x),\; \text{in}\;\partial\Omega.
		\end{aligned}
	\end{equation}
    The unknowns are the vector function $\boldsymbol{u}$(velocity) and the scalar function $p$(pressure), $\boldsymbol{f}=\boldsymbol{f}(x,t)$ is the known body force and $\nu$ is the kinematic viscosity.  \\
    \indent We consider the natural velocity space by $\boldsymbol{X} = \boldsymbol{H}^1_0(\Omega)$ and pressure space by $Q = L^2_0(\Omega)$, functions with vanishing integral on $\Omega$. In $H^1_0(\Omega)$, we have the Poincaré inequality \cite[p.135]{Brenner2008}: there exists a constant $C_{P}$ depending only on $\Omega$ such that for $\forall v\in H^1_0(\Omega)$,
    \begin{equation}\label{PoinIneq}
    	\begin{aligned}
    		    \|v\| \leq C_P\|\nabla v\|.
        \end{aligned}
    \end{equation}
    This inequality makes the semi-norm $|v|_1 := \|\nabla v\|$ in $H^1(\Omega)$ is a norm in $H^1_0(\Omega)$. Let ${\mathcal{T}_h}$ be a uniformly regular family of triangulation of $\overline{\Omega}$ in the sense of \cite[p.111]{Ciarlet-FEM-2002} and $h := \max_{K\in \mathcal{T}_h}\{h_K | h_K := \text{diam}(K)\}$. $(\boldsymbol{X}_h, Q_h) \subset (\boldsymbol{X}, Q)$ denotes the conforming and inf–sup stable finite element (FE) spaces. We herein consider $(\boldsymbol{X}_h, Q_h) = (\mathbb{\boldsymbol{P}}_k, \mathbb{P}_{k-1})$, and in our computation $k=2$. By the quasi-uniform of mesh, the following inverse inequality holds for $\forall \boldsymbol{v}_h\in \boldsymbol{X}_h$
    \begin{equation}\label{InveIneq}
    	\left\|\boldsymbol{v}_h\right\|_{W^{m, p}(K)} \leq C_{\mathrm{inv}} h_K^{n-m-d\left(\frac{1}{q}-\frac{1}{p}\right)}\left\|\boldsymbol{v}_h\right\|_{W^{n, q}(K)},
    \end{equation}
    where $0 \leq  n \leq  m \leq  1, 1 \leq  q \leq  p \leq  \infty $. This inequality also holds for $\forall q_h \in Q_h$. The convection term is adopted as the skew-symmetric convection form
    $$
    b(u, v, w) := \frac{1}{2}(u \cdot \nabla v, w) - \frac{1}{2}(u \cdot \nabla w, v) \quad \forall u, v, w \in X.
    $$
    An equivalent form of $b$ can be obtained immediately via
    $$
    b(u, v, w) = (u \cdot \nabla v, w) - \frac{1}{2}((\nabla\cdot u)  v, w) \quad \forall u, v, w \in X,
    $$
    and we will utilize the skew-symmetry property of $b$ as $b(u, v, v) = 0$ for $u,v\in X$. \\
	\indent In order to construct ROM, we use the finite element method(FEM) and semi-implicit backward differentiation formula of order 1(BDF1) for spatial and temporal to discretize the continuous variational NSE: For any $t > 0$, find $(\boldsymbol{u}(t),p(t)) \in (\boldsymbol{X},Q)$, such that for $\forall (\boldsymbol{v},q) \in (\boldsymbol{X},Q)$,
	\begin{equation}\label{ContVariNSEs}
		\left\{
		\begin{aligned}
			\left(\partial_t\boldsymbol{u}, \boldsymbol{v}\right) + \nu\left(\nabla \boldsymbol{u}, \nabla \boldsymbol{v}\right) + b\left(\boldsymbol{u},\boldsymbol{u},\boldsymbol{v}_h\right) - \left(p, \nabla \cdot\boldsymbol{v}\right) &=\left(f, \boldsymbol{v}\right), \\
			\left(\nabla \cdot \boldsymbol{u}, q\right) &= 0,
		\end{aligned}
		\right.
	\end{equation}
	to obtain the following scheme, denoted as the full-order model(FOM): Given initial condition $(\boldsymbol{u}_h^0,p^0_h,p^{-1}_h)$, where $\boldsymbol{u}_h^0 = \Pi \boldsymbol{u}_0$, $\Pi$ could be Lagrange interpolation or $L^2$ projection operator into FE space, and total numbers of time-steps $N$. Then, for $n=1,2,\cdots, N\!-\!1$ and $\forall (\boldsymbol{v}_h,q_h) \in (\boldsymbol{X}_h, Q_h)$, find $(\boldsymbol{u}^{n+1}_h,p^{n+1}_h) \in (\boldsymbol{X}_h, Q_h)$ satisfying 
	\begin{equation}\label{FEM-FOM}
		\left\{
		\begin{aligned}
			\frac{1}{\Delta t}\left(\boldsymbol{u}_h^{n+1}-\boldsymbol{u}_h^n, \boldsymbol{v}_h\right) + \nu\left(\nabla \boldsymbol{u}_h^{n+1}, \nabla \boldsymbol{v}_h\right) + b\left(\boldsymbol{u}_h^n,\boldsymbol{u}_h^{n+1},\boldsymbol{v}_h\right) + \left(\nabla\left(2 p_h^n-p_h^{n-1}\right), \boldsymbol{v}_h\right) &=\left(f^{n+1}, \boldsymbol{v}_h\right), \\
			\left(\nabla \cdot \boldsymbol{u}_h^{n+1}, q_h\right) + \Delta t\left(\nabla\left(p_h^{n+1}-p_h^n\right), \nabla q_h\right) &= 0.
		\end{aligned}
		\right.
	\end{equation}
    \begin{Remark}
    	For the choice of $p^{-1}_h$, we follow the suggestion in \cite{Proj-Guermond-1998-IJNMF} to select $p^{-1}_h = p^0_h$ in the computation.
    \end{Remark}
    \indent Given the choice that we will use the FE solutions as the snapshots to construct ROM in the forthcoming section, the approximation error between the exact solution and reduced-order solution, in this case, will consist of three parts: spatial, temporal and ROM discretization errors. To analyze the convergence of approximation, we need the following regularity assumptions on the NSE, which may not be valid in some realistic situations, however, we still do it simply for the sake of theoretical analysis.
    \begin{Assumption}
    	We assume that the solution of the NSE satisfies
    	$$
    	\begin{aligned}
    		& \boldsymbol{u} \in L^{\infty}\left(0, T ; \boldsymbol{H}^1(\Omega)\right) \cap H^1\left(0, T ; \boldsymbol{H}^{k+1}(\Omega)\right) \cap H^2\left(0, T ; \boldsymbol{H}^1(\Omega)\right) \\
    		& p \in L^2\left(0, T ; H^{k+1}(\Omega)\right), \; \boldsymbol{f} \in L^2\left(0, T ; \boldsymbol{L}^2(\Omega)\right) .
    	\end{aligned}
        $$
    \end{Assumption}
    Based on the above assumption about continuous solutions, the following lemma represents the first two errors which has been analyzed and verified in FE community, e.g. \cite{Proj-Guermond-1998-IJNMF}.
    \begin{Lemma}[\!\!\cite{Proj-Guermond-1998-IJNMF}]\label{Conv-Exact-FE}
    	The FE solution in FOM-FOM (\ref{FEM-FOM}) satisfies the following optimal error estimates.
    	\begin{equation}
    	\begin{aligned}
    		\max _{0 \leq n \leq N}\left\|\boldsymbol{u}_h^n-\boldsymbol{u}\left(t^n\right)\right\| &\leq C(\boldsymbol{u}, p)\left(\Delta t+h^{k+1}\right), \\
    		\max _{0 \leq n \leq N}\left\|\boldsymbol{u}_h^n-\boldsymbol{u}\left(t^n\right)\right\|_1 + \max _{0 \leq n \leq N}\left\|p_h^n-p\left(t^n\right)\right\| &\leq C(\boldsymbol{u}, p)\left(\Delta t+h^{k}\right).
    	\end{aligned}
        \end{equation}
    \end{Lemma}
    \indent The following symbols and lemmas will be used later
    $$
    \|\cdot\|_{\ell^2(L^2)} := \left(\Delta t\sum\|\cdot\|^2\right)^{\frac12}, \quad \|\cdot\|_{L^2(L^2)} := \left(\int^{T}_{0}\|\cdot\|^2\right)^{\frac12}
    $$
	\indent In the forthcoming convergence analysis about reduced-order solutions, we will bound $\|u^n_h\|_{W^{s,\infty}}$, $s=0,1$. To this end, we use modified Stokes projection $s_h$ to serve as an intermediate variable, which, compared to the classical Stokes projection\cite{Heywood-Rannacher-1982-SINUM-3}, is introduced in \cite{Frutos-ModifiedStokesProj-2016-JSC} and aims to avoid the appearance of $1/\nu$ in the upper-bound of projection error. Specifically, let $(u,p)$ be the solution of the continuous Navier-Stokes equation (\ref{NSE}) with $\boldsymbol{u}\in \boldsymbol{V}\cap \boldsymbol{H}^{k+1}$, $p\in Q\cap H^{k}$, $k\geq 1$, for $t\geq 0$, thus the pair $(\boldsymbol{u},0)$ is the solution of the Stokes problem
	$$
	\left\{
	\begin{aligned}
		-\nu\Delta \boldsymbol{\tilde{u}} + \nabla \tilde{p} &= g, \\
		\nabla\cdot \boldsymbol{\tilde{u}} &= 0.
	\end{aligned}\right.
	$$
	with the right-hand side 
	$$
	\boldsymbol{g} = \boldsymbol{f} - \partial_t \boldsymbol{u} - \boldsymbol{u}\cdot\nabla\boldsymbol{u} - \nabla p.
	$$
	\indent Multiplying test function by the above Stokes equation, integrating by part and discretizing by inf-sup stable mixed finite element $(\boldsymbol{X}_h,Q_h)$, we get the discrete solution and denote the velocity component as $\boldsymbol{s}_h$; then, the following bounds and stability hold
	\begin{Lemma}[Error and stability estimates for modified Stokes projection\cite{POD-LPS-SINUM-2021}]
		For $\forall t \geq 0$, $(u(t),p(t))$ denote the solutions of NSEs (\ref{NSE}), $\boldsymbol{s}_h$ is the mixed FE velocity solution defined above. The following bounds are hold
		$$
		\| \boldsymbol{u} - \boldsymbol{s}_h \|_1 \leq Ch^k\|\boldsymbol{u}\|_{k+1}.
		$$
		and 
		$$
		\begin{aligned}
			\left\|\boldsymbol{s}_h\right\|_{\infty} & \leq C_{1,S}(\|\boldsymbol{u}\|\|\boldsymbol{u}\|_2)^{1/2}. \\
			\left\|\nabla \boldsymbol{s}_h\right\|_{\infty} & \leq C_{2,S}\|\nabla \boldsymbol{u}\|_{\infty}.
		\end{aligned}
	    $$
	    where three constants $C,C_{1,S},C_{2,S}$ are all independent on $\nu$.
	\end{Lemma}

    \begin{Lemma}[Discrete Gronwall Lemma\cite{Heywood-Rannacher-1982-SINUM-4}]\label{Gronwall}
    	Let $k, B$, and $a_j, b_j, c_j, \gamma_j$, for integers $j \geq 0$, be non-negative numbers such that
    $$
    a_n+k \sum_{j=0}^n b_j \leq k \sum_{j=0}^n \gamma_j a_j+k \sum_{j=0}^n c_j+B \quad \text { for } n \geq 0 .
    $$
    Suppose that $k \gamma_j<1$, for all $j$, and set $\sigma_j = \left(1-k \gamma_j\right)^{-1}$. Then,
    $$
    a_n+k \sum_{j=0}^n b_j \leq \exp \left(k \sum_{j=0}^n \sigma_j \gamma_j\right)\left\{k \sum_{j=0}^n c_j+B\right\} \quad \text { for } n \geq 0.
    $$
    \end{Lemma}

    \subsection{ROM preliminaries}
    \indent We will consider in this manuscript the proper orthogonal decomposition(POD) technique to generate the reduced-order solutions. To this end, we need to first run the FEM-FOM(\ref{FEM-FOM}) on some given time interval and thus collect its high-fidelity solutions to form the velocity and pressure snapshot matrices 
    $$
    \mathcal{U} :=\left\langle\boldsymbol{u}_h^{1}, \boldsymbol{u}_h^{2}, \cdots, \boldsymbol{u}_h^{M}\right\rangle,\quad\mathcal{P} :=\left\langle p_h^{1}, p_h^{2}, \cdots, p_h^{M}\right\rangle,
    $$
    for $M$ different discrete time instances and the two dimensions are $d_u,d_p$. It is also a common practice to include in the snapshots matrices the difference quotients(DQs) $\partial f^{i} := (f^i - f^{i-1})/\Delta t$, $i=2,3,\cdots,M$, for some function $f$ at discrete time instance $t=t_i$. For brevity, we only consider the case of snapshots without DQs. As for the effects on ROM of adding DQs, we refer to \cite{Iliescu-Wang-2014-SISC}; the necessity of having DQs for obtaining the piece-wise optimal projection error has been confirmed in \cite{Koc-TimePointwise-SINUM-2021}. We adopt $L^2$ inner product for both velocity and pressure to formulate correlation matrices $K_u=\left(\left(k^u_{i, j}\right)\right) \in \mathbb{R}^{M \times M}$ and $K_p=\left(\left(k^p_{i, j}\right)\right) \in \mathbb{R}^{M \times M}$ where:
    $$
    k^u_{i, j} := \frac{1}{M}(\boldsymbol{u}_h^{i}, \boldsymbol{u}_h^{j}),\quad k^p_{i, j} := \frac{1}{M}(\boldsymbol{p}_h^{i}, \boldsymbol{p}_h^{j}).
    $$
    The POD aims to find a low-dimensional basis that approximates these snapshots optimally with respect to least-square residual norm. Following \cite{Galerkin-POD-NM-2001}, a singular value decomposition(SVD) is carried out and the leading generalized eigenfunctions are chosen as bases, referred to as the POD bases. We denote by $\lambda_1 \geq \lambda_2 \geq \cdots \geq \lambda_{d_u}>0$ the positive eigenvalues of $K_u$ and $\boldsymbol{x}_1, \boldsymbol{x}_2, \ldots, \boldsymbol{x}_{d_{u}} \in \mathbb{R}^{M}$ the associated eigenvectors. Analogously, $\left\{\theta_i, \boldsymbol{y}_i\right\}_{i=1}^{d_p}$ are the eigen-pairs of $K_p$. Then, the (orthonormal) POD basis can be written explicitly as
    $$
    \boldsymbol{\varphi}_k = \frac{1}{\sqrt{M\lambda_k}} \sum_{j=1}^M x_k^j \boldsymbol{u}^j_h,\quad
    \psi_k = \frac{1}{\sqrt{M\theta_k}} \sum_{j=1}^M y_k^j p^j_h.
    $$
    We choose $r_u=r_p=r$ for simplicity, dimensions of reduced-order spaces, to span the reduced-order velocity space and pressure space
    $$
    \boldsymbol{X}_r=\left\langle\boldsymbol{\varphi}_1, \boldsymbol{\varphi}_2, \cdots, \boldsymbol{\varphi}_r\right\rangle, \quad Q_r=\left\langle\psi_1, \psi_2, \cdots, \psi_r\right\rangle.
    $$
    \indent After obtaining the low-dimensional spaces above, the FEM-FOM(\ref{FEM-FOM}) is projected by Galerkin projection into $(\boldsymbol{X}_r,Q_r)$ to gain the following plain POD-ROM(G-POD-ROM):
    \begin{equation}\label{G-POD-ROM}
    	\left\{
    	\begin{aligned}
    		\frac{1}{\Delta t}\left(\boldsymbol{u}_r^{n+1}-\boldsymbol{u}_r^n, \boldsymbol{v}_r\right) + \nu\left(\nabla \boldsymbol{u}_r^{n+1}, \nabla \boldsymbol{v}_r\right) + b\left(\boldsymbol{u}_r^n,\boldsymbol{u}_r^{n+1},\boldsymbol{v}_r\right) + \left(\nabla\left(2 p_r^n-p_r^{n-1}\right), \boldsymbol{v}_r\right) &=\left(\boldsymbol{f}^{n+1}, \boldsymbol{v}_r\right), \\
    		\left(\nabla \cdot \boldsymbol{u}_r^{n+1}, q_r\right) + \Delta t\left(\nabla\left(p_r^{n+1}-p_r^n\right), \nabla q_r\right) &= 0.
    	\end{aligned}
    	\right.
    \end{equation}
    
    \begin{Definition}
    	We define the following $L^2$ orthogonal projection operator $\Pi^u_r,\Pi^p_r$
    	$$
    	\begin{aligned}
    		(\boldsymbol{u} - \Pi^u_r\boldsymbol{u}, \boldsymbol{v}_r) = 0,\;\forall \boldsymbol{v}_r\in \boldsymbol{V}_r, \quad (p - \Pi^p_rp, q_r) = 0,\;\forall q_r\in Q_r.
    	\end{aligned}
    	$$
    \end{Definition}
    \indent The $W^{s,\infty}$, $s=0,1$ stability can be obtained by the following lemma whose proof is given in Appendix.
    \begin{Lemma}\label{Lemma-Stab}
    	The $L^2$ orthogonal projection operator $\Pi^u_r$ defined above is stable in $W^{s,\infty}$, $s=0,1$, for $\forall \boldsymbol{v}_h\in \boldsymbol{X}_h$
    	\begin{equation}\label{Stab-L2-Piu}
    		\|\Pi^u_r\boldsymbol{v}_h\|_{L^{\infty}} \leq C_{stab,0} \quad\text{and}\quad \|\nabla\Pi^u_r\boldsymbol{v}_h\|_{L^{\infty}} \leq C_{stab,1}.
    	\end{equation}
    \end{Lemma}
    \begin{proof}
    	See Appendix \ref{Proof-Stab}.
    \end{proof}
    \indent As we have pointed out at the beginning of this section, the existence of DQs can justify the optimal piece-wise projection error\cite{Koc-TimePointwise-SINUM-2021}, which has been an assumption for the past two decades. We herein make the same assumption which can be justified with DQs, and also the appearance of DQs would not damage the theoretical and numerical results within this manuscript.
    \begin{Assumption}\label{Assu}
    	We have the following piece-wise in time projection error
    	\begin{equation}
    	\begin{aligned}		
    		\|\boldsymbol{u}^n_h - \Pi^u_r\boldsymbol{u}^n_h\|^2 &\leq C\sum_{i=r_u+1}^{d_u}\lambda_i, \quad \|\nabla(\boldsymbol{u}^n_h - \Pi^u_r\boldsymbol{u}^n_h)\|^2 \leq C\sum_{i=r_u+1}^{d_u}\lambda_i\|\nabla\varphi_i\|^2, \\
    		\|p^n_h - \Pi^p_rp^n_h\|^2 &\leq C\sum_{i=r_p+1}^{d_p}\theta_i, \quad \|\nabla(p^n_h - \Pi^p_rp^n_h)\|^2 \leq C\sum_{i=r_p+1}^{d_p}\theta_i\|\nabla\psi_i\|^2.
    	\end{aligned}
        \end{equation}
    \end{Assumption}
    
    \subsection{Data assimilation preliminaries}
    We consider $I_H$ to be an interpolation operator that satisfies: For a given mesh ${\mathcal{T}_H}$ with $H \leq 1$,
    \begin{align}
    	\left\|I_H(w)-w\right\| & \leq C_{1,H} H\|\nabla w\|, \label{Conv-H} \\ 
    	\left\|I_H(w)\right\| & \leq C_{2,H}\|w\|,  \label{Stab-H}
    \end{align}
    for any $w \in H^1(\Omega)$. For example, this holds for the $L^2$ projection onto piecewise constants, and also for the Scott-Zhang interpolant or Lagrange interpolation for smooth functions. For the (unknown) true solution $u$, $I_H(u)$ represents an approximation of what is observed of the true solution. We assume in this paper that $I_H(u)$ can be observed at any time. We note that, at the computational level, $I_H$ could be adopted Cl$\rm{\acute{e}}$ment interpolation used in \cite{CDA-ACM-2020}, Lagrange interpolation in \cite{CDA-SINUM-2020}, $L^2$ projection in \cite{CDA-VeloVort-2021-ERA}, and the modified $L^2$ projection in \cite{CDA-2021-NMPDE}, all of those operators are operated into piece-wise constant space. In the forthcoming numerical experiment, we adopt the last one mentioned above.

    \section{CDA-POD-ROM and Analysis}\label{section-3}
    \indent Our newly proposed continuous data assimilation POD-ROM incorporates a feedback control(penalty) term into G-POD-ROM(\ref{G-POD-ROM}), aiming at nudging the reduced-order solution $(\boldsymbol{u}_r^{n+1}, p_r^{n+1})$ toward the reference solution $(\boldsymbol{u}^{n+1}, p^{n+1})$ corresponding to the observation data. More specifically, a velocity penalty term is merged into the momentum equation in G-POD-ROM(\ref{G-POD-ROM}) to increase the accuracy of reduced-order velocity, and we add a similar term to the artificial compressible equation with a similar purpose but appear as a phenomenon that we can directly solve and obtain a table and accurate reduced-order pressure coupling the reduced-order velocity, rather than recovering pressure \textit{a posterior} by some pressure recovery techniques, e.g. pressure Poisson equation\cite{NS-POD-2014-JCP}. That is, the continuous data assimilation pressure-correction projection POD-ROM(CDA-Proj-POD-ROM) takes the form: Find $(\boldsymbol{u}_r^{n+1}, p_r^{n+1}) \in (\boldsymbol{X}_r, Q_r)$, such that for $\forall (\boldsymbol{v}_r, q_r) \in (\boldsymbol{X}_r, Q_r)$, we have
    \begin{equation}\label{CDA-Proj-POD-ROM}
    	\left\{
    	\begin{aligned}
    		\frac{1}{\Delta t}\left(\boldsymbol{u}_r^{n+1}-\boldsymbol{u}_r^n, \boldsymbol{v}_r\right) &+ \nu\left(\nabla \boldsymbol{u}_r^{n+1}, \nabla \boldsymbol{v}_r\right) + b\left(\boldsymbol{u}_r^n,\boldsymbol{u}_r^{n+1},\boldsymbol{v}_r\right) \\
    		&+ \left(\nabla\left(2 p_r^n-p_r^{n-1}\right), \boldsymbol{v}_r\right) + \gamma_u(I_H(\boldsymbol{u}_r^{n+1} - \boldsymbol{u}^{n+1}), I_H\boldsymbol{v}_r) =\left(\boldsymbol{f}^{n+1}, \boldsymbol{v}_r\right), \\
    		\left(\nabla \cdot \boldsymbol{u}_r^{n+1}, q_r\right) &+ \Delta t\left(\nabla\left(p_r^{n+1}-p_r^n\right), \nabla q_r\right) + \gamma_p(I_H(p_r^{n+1} - p^{n+1}), I_Hq_r) = 0,
    	\end{aligned}
    	\right.
    \end{equation}
	for $n=n_0,n_0+1,\cdots N$, $n_0 \geq 1$, and the initial conditions given by Galerkin projection as $(\boldsymbol{u}_r^{n_0}, p_r^{n_0}, p_r^{n_0-1}) = (\Pi^u_r\boldsymbol{u}_r^{n_0}, \Pi^p_rp_r^{n_0}, \Pi^p_rp_r^{n_0-1})$, and the $\gamma_u,\gamma_p$ are nudging parameters corresponding to velocity and pressure respectively(resp.).  \\
	\indent The test functions $I_H\boldsymbol{v}_r$(resp. $I_Hq_r$) in velocity nudging term(resp. in pressure's) defined above is not exactly the same as the one, $\boldsymbol{v}_r$, in the paper which proposed the CDA\cite{CDA-2014-JNS}, but follows the suggestion in \cite{CDA-2021-NMPDE} for efficient implementation and long-term stability. This modification does not damage the convergence as already adopted and analyzed in \cite{CDA-2019-CMAME,CDA-ACM-2020,CDA-SINUM-2020} etc. 

	\subsection{Stability}
	\indent In this subsection, we will derive some numerical analysis about the CDA-Proj-POD-ROM(\ref{CDA-Proj-POD-ROM}). Firstly, we can get the following unconditional stability about the reduced-order velocity and pressure.
	\begin{Theorem}
		We assume the initial solutions are stable, then for $n\geq n_0$ and $\forall \gamma_u,\gamma_p \geq 0$, we have
		$$
		\begin{aligned}
			\|\boldsymbol{u}_r^{n+1}\|^2 + \Delta t^2\|\nabla p^{n+1}_r\|^2 + \nu\Delta t\sum_{j=1}^{n}\|\nabla \boldsymbol{u}^{j+1}_r\|^2 \leq C_{init} + C_{cont},
		\end{aligned}
		$$
		where the initial data and continuous data are defined 
		$$
		\begin{aligned}
		C_{init} &:= \|\boldsymbol{u}_r^{1}\|^2 + \Delta t^2\|\nabla p^{1}_r\|^2 + 1/2\Delta t\gamma_p C^2_{2,H}\|p^1_r - p^0_r\|^2,  \\
		C_{cont} &:=  T\Big[C^2_{P}/\nu\|\boldsymbol{f}\|_{L^{\infty}(0,T;\boldsymbol{L}^2)} + C^2_{2,H}\Big(6\gamma_p\|p\|_{L^{\infty}(0,T;L^2)} + \gamma_u\|\boldsymbol{u}\|_{L^{\infty}(0,T;\boldsymbol{L}^2)}\Big)  \Big].
	    \end{aligned}
		$$
	\end{Theorem}
	\begin{proof}
		Setting $(\boldsymbol{v}_r,q_r) = (\boldsymbol{u}^{n+1}_r, 2p^n_r - p^{n-1}_r)$ and using skew-symmetry property, scheme (\ref{CDA-Proj-POD-ROM}) becomes
		$$
		\begin{aligned}
			\frac{1}{\Delta t}(\boldsymbol{u}^{n+1}_r &- \boldsymbol{u}^{n}_r, \boldsymbol{u}^{n+1}_r) + \nu\|\nabla \boldsymbol{u}^{n+1}_r\|^2 + \Delta t\left(\nabla(p^{n+1}_r - p^{n}_r), \nabla(2p^{n}_r - p^{n-1}_r)\right) \\
			&+ \gamma_u(I_H(\boldsymbol{u}_r^{n+1} - \boldsymbol{u}^{n+1}), I_H\boldsymbol{u}^{n+1}_r) + \gamma_p\left(I_H(p^{n+1}_r - p^{n+1}), I_H(2p^{n}_r - p^{n-1}_r)\right) = (\boldsymbol{f}^{n+1}, \boldsymbol{u}^{n+1}_r).
		\end{aligned}
		$$
		This is equivalent to
		$$
		\begin{aligned}
			\frac{1}{\Delta t}\left(\boldsymbol{u}^{n+1}_r - \boldsymbol{u}^{n}_r, \boldsymbol{u}^{n+1}_r\right) &+ \nu\|\nabla \boldsymbol{u}^{n+1}_r\|^2 + \Delta t\left(\nabla(p^{n+1}_r - p^{n}_r), \nabla p^{n+1}_r\right) + \gamma_p\left(I_H(p^{n+1}_r - p^{n+1}), I_H(p^{n+1}_r)\right) \\
			&+ \gamma_u(I_H(\boldsymbol{u}_r^{n+1} - \boldsymbol{u}^{n+1}), I_H\boldsymbol{u}^{n+1}_r) \\
			=& (\boldsymbol{f}^{n+1}, \boldsymbol{u}^{n+1}_r) + \Delta t\left(\nabla(p^{n+1}_r - p^{n}_r), \nabla(p^{n+1}_r - p^{n}_r) - \nabla(p^{n}_r - p^{n-1}_r)\right) \\
			&+ \gamma_p\left(I_H(p^{n+1}_r - p^{n+1}), I_H(p^{n+1}_r - p^{n}_r) - I_H(p^{n}_r - p^{n-1}_r)\right) \\
			:=& S_1 + S_2 + S_3.
		\end{aligned}
		$$
		For $S_1$,
		$$
		|S_1| \leq \frac{C^2_{P}}{2\nu}\|\boldsymbol{f}^{n+1}\|^2 + \frac12\nu\|\nabla \boldsymbol{u}^{n+1}_r\|^2,
		$$
		and for $S_2$, using the second equation in the scheme by letting $q_r=p^{n+1}_r-p^n_r$ at $t=t_{n+1}$ and $t=t_n$ separately to get
		$$
		\begin{aligned}
			\left(\nabla \cdot (\boldsymbol{u}_r^{n+1} - \boldsymbol{u}_r^{n}), p^{n+1}_r-p^n_r\right) &+  \Delta t\left(\nabla\left(p_r^{n+1}-p_r^n\right) - \nabla\left(p_r^{n}-p_r^{n-1}\right), \nabla (p^{n+1}_r-p^n_r)\right) \\
			&+ \gamma_p(I_H(p_r^{n+1} - p^{n+1}) - I_H(p_r^{n} - p^{n}), I_H(p^{n+1}_r-p^n_r)) = 0,
	    \end{aligned}
		$$
		which means,
		$$
		\begin{aligned}
			S_2 &= -\left(\nabla \cdot (\boldsymbol{u}_r^{n+1} - \boldsymbol{u}_r^{n}), p^{n+1}_r-p^n_r\right) - \gamma_p(I_H(p_r^{n+1} - p^{n+1}) - I_H(p_r^{n} - p^{n}), I_H(p^{n+1}_r-p^n_r)) \\
			:&= S_{21} + S_{22}.
    	\end{aligned}
        $$	
		Thus, 
		$$
		\begin{aligned}
			|S_{21}| &= \left|\left(\boldsymbol{u}_r^{n+1} - \boldsymbol{u}_r^{n}, \nabla(p^{n+1}_r-p^n_r)\right)\right| \leq \frac{1}{2\epsilon_{1}}\|\boldsymbol{u}_r^{n+1} - \boldsymbol{u}_r^{n}\|^2 + \frac{\epsilon_{1}}{2}\|\nabla p^{n+1}_r - \nabla p^n_r\|^2,
		\end{aligned}
		$$	
		and
		$$
		\begin{aligned}
			S_{22} &= - \gamma_p(I_H(p_r^{n+1} - p_r^{n}) - I_H(p^{n+1} - p^{n}), I_H(p^{n+1}_r-p^n_r)) \\
			 &= -\frac12\gamma_p\bigg(\|I_H(p^{n+1}_r-p^n_r)\|^2 - \|I_H(p^{n+1} - p^{n})\|^2 + \|I_H(p_r^{n+1} - p_r^{n}) - I_H(p^{n+1} - p^{n})\|^2\bigg).
		\end{aligned}
		$$
		Moreover, 
		$$
		\begin{aligned}
			S_3 &= \gamma_p\left(I_H(p^{n+1}_r - p^{n+1}), I_H(p^{n+1}_r - p^{n}_r) - I_H(p^{n}_r - p^{n-1}_r)\right) \\
			&\leq \frac{\gamma_p}{2\epsilon_{2}}\|I_H(p^{n+1}_r - p^{n+1})\|^2 + \frac{\epsilon_{2}}{2}\gamma_p\|I_H(p^{n+1}_r - p^{n}_r) - I_H(p^{n}_r - p^{n-1}_r)\|^2 \\
			&\leq \frac{\gamma_p}{2\epsilon_{2}}\|I_H(p^{n+1}_r - p^{n+1})\|^2 + \frac{\epsilon_{2}}{2}\gamma_p\left(\|I_H(p^{n+1}_r - p^{n}_r)\|^2 + \|I_H(p^{n}_r - p^{n-1}_r)\|^2\right)
		\end{aligned}
		$$
		Collecting all results above and rearranging to get
		$$
		\begin{aligned}
			\frac{1}{2\Delta t}(&\|\boldsymbol{u}_r^{n+1}\|^2 - \|\boldsymbol{u}_r^{n}\|^2) + \left(\frac{1}{2\Delta t} - \frac{1}{2\epsilon_{1}} \right)\|\boldsymbol{u}_r^{n+1} - \boldsymbol{u}_r^{n}\|^2 + \frac12\nu\|\nabla \boldsymbol{u}^{n+1}_r\|^2 + \frac12\Delta t\left(\|\nabla p^{n+1}_r\|^2 - \|\nabla p^{n}_r\|^2\right) \\
			&+ \left(\frac12\Delta t - \frac{\epsilon_{1}}{2}\right)\|\nabla p^{n+1}_r - \nabla p^{n}_r\|^2 + \left(1 - \epsilon_{2}\right)\frac12\gamma_p\|I_H(p^{n+1}_r-p^n_r)\|^2 \\
			&+ \frac12\gamma_p\|I_H(p_r^{n+1} - p_r^{n}) - I_H(p^{n+1} - p^{n})\|^2 + \left(\frac12 - \frac{1}{2\epsilon_{2}}\right)\gamma_p\|I_Hp^{n+1}_r - I_H p^{n+1}\|^2 \\
			&+ \frac12\gamma_p\left(\|I_Hp^{n+1}_r\|^2 - \| I_H p^{n+1}\|^2\right) + \frac12\gamma_u\|I_H\boldsymbol{u}_r^{n+1}\|^2 + \frac12\gamma_u\|I_H\boldsymbol{u}^{n+1}_r - I_H\boldsymbol{u}^{n+1}\|^2 \\
		\leq& \frac{C^2_{P}}{2\nu}\|\boldsymbol{f}^{n+1}\|^2 + \frac{\epsilon_{2}}{2}\gamma_p\|I_H(p^{n}_r - p^{n-1}_r)\|^2 + \frac12\gamma_p\|I_H(p^{n+1} - p^{n})\|^2 + \frac12\gamma_u\|I_H\boldsymbol{u}^{n+1}\|^2.
		\end{aligned}
		$$
		Those terms on the LHS need to remain positive, to this end, we let $\epsilon_{1} = \Delta t$, and $\epsilon_{2} = 1/2$, then we bound the negative term on LHS as
		$$
		\begin{aligned}
		\left(\frac12 - \frac{1}{2\epsilon_{2}}\right)\gamma_p\|I_Hp^{n+1}_r - I_H p^{n+1}\|^2	&\geq -\frac12\gamma_p\|I_Hp^{n+1}_r\| -\frac12\gamma_p\|I_H p^{n+1}\|^2,
    	\end{aligned}
		$$
		and using the estimate 
		$$
		(a+b)^2 \leq 2(a^2 + b^2) \quad \Rightarrow \quad(a-b)^2 \geq \frac12 a^2 - b^2,
		$$
		to deduce
		$$
		\frac12\gamma_p\|I_H(p_r^{n+1} - p_r^{n}) - I_H(p^{n+1} - p^{n})\|^2 \geq \frac14\gamma_p\|I_H(p_r^{n+1} - p_r^{n})\|^2 - \frac12\gamma_p\|I_H(p^{n+1} - p^{n})\|^2.
		$$
		Then, the above inequality becomes
		$$
		\begin{aligned}
			\frac{1}{2\Delta t}\Big(&\|\boldsymbol{u}_r^{n+1}\|^2 - \|\boldsymbol{u}_r^{n}\|^2\Big) + \frac12\nu\|\nabla \boldsymbol{u}^{n+1}_r\|^2 + \frac12\Delta t\Big(\|\nabla p^{n+1}_r\|^2 - \|\nabla p^{n}_r\|^2\Big) \\
			&+ \frac14\gamma_p\Big(\|I_H(p_r^{n+1} - p_r^{n})\|^2 - \|I_H(p^{n}_r - p^{n-1}_r)\|^2\Big) \\
			&+ \frac12\gamma_u\|I_H\boldsymbol{u}_r^{n+1}\|^2 + \frac12\gamma_u\|I_H\boldsymbol{u}^{n+1}_r - I_H\boldsymbol{u}^{n+1}\|^2 \\
			\leq& \frac{C^2_{P}}{2\nu}\|\boldsymbol{f}^{n+1}\|^2 + \gamma_p\|I_H(p^{n+1} - p^{n})\|^2 + \gamma_p\| I_H p^{n+1}\|^2 
			+ \frac12\gamma_u\|I_H\boldsymbol{u}^{n+1}\|^2.
		\end{aligned}
		$$
		Multiplying $2\Delta t$, adding $n$ from $1$ to $n$, we get
		$$
		\begin{aligned}
			\Big(\|\boldsymbol{u}_r^{n+1}\|^2 &+ \Delta t^2\|\nabla p^{n+1}_r\|^2 + \frac12\Delta t\gamma_p\|I_H(p^{n+1}_r-p^{n}_r)\|^2\Big) \\
			& + \Delta t\sum_{j=1}^{n}\Big(\nu\|\nabla \boldsymbol{u}^{j+1}_r\|^2 + \gamma_u\|I_H\boldsymbol{u}_r^{j+1}\|^2 + \gamma_u\|I_H\boldsymbol{u}^{j+1}_r - I_H\boldsymbol{u}^{j+1}\|^2 \Big) \\
			\leq& \left(\|\boldsymbol{u}_r^{1}\|^2 + \Delta t^2\|\nabla p^{1}_r\|^2 + \frac12\Delta t\gamma_p\|I_H(p^{1}_r-p^0_r)\|^2\right) \\
			&+ \Delta t\sum_{j=1}^{n}\left(\frac{C^2_{P}}{\nu}\|\boldsymbol{f}^{j+1}\|^2 + 2\gamma_p\| I_H p^{j+1}\|^2 + 2\gamma_p\|I_H(p^{j+1} - p^{j})\|^2 + \gamma_u\|I_H\boldsymbol{u}^{j+1}\|^2\right)
		\end{aligned}
		$$
		Dropping unneeded terms, using stability estimate about $I_H$ and the regularity assumptions, we complete the proof.
	\end{proof}
	\subsection{Convergence}
	\indent We next prove that the reduced-order solutions in (\ref{CDA-Proj-POD-ROM}) converge to the high-fidelity solutions in (\ref{FEM-FOM}), up to disretization and ROM projection error, and then with which we could conclude to the convergence of reduced-order solutions, with the help of the convergence results Lemma \ref{Conv-Exact-FE} of high-fidelity solutions.  \\
	\indent We firstly take symbols as  
	$$
	\begin{aligned}
		\boldsymbol{e}^n_u &:= \boldsymbol{u}^n_h - \boldsymbol{u}^n_r = \Big(\boldsymbol{u}^n_h - \Pi^u_r\boldsymbol{u}^n_h\Big) + \Big(\Pi^u_r\boldsymbol{u}^n_h - \boldsymbol{u}^n_r\Big) \triangleq \boldsymbol{\eta}^n_u + \boldsymbol{\xi}^n_u, \;\text{and}	\\
		e^n_p &:= p^n_h - p^n_r = \Big(p^n_h - \Pi^p_rp^n_h\Big) + \Big(\Pi^p_rp^n_h - p^n_r\Big) \triangleq \eta^n_p + \xi^n_p,
	\end{aligned}
	$$
	then, we can obtain
	\begin{Lemma}\label{Conv-FE-POD}
		For any $n \geq n_0$, $n_0 \geq 1$, we assume $\Delta t \leq \mathcal{O}(h^2)$ and for some constant $C$, we have the following convergence
		\begin{equation}
			\begin{aligned}	
				\|\boldsymbol{e}^{n+1}_u\|^2 & + \nu\Delta t\|\nabla \boldsymbol{e}^{n+1}_u\|^2 + \Delta t^2\|\nabla e^{n+1}_p\|^2 + \frac14\gamma_p\Delta t\sum_{j=n_0}^{n}\|e^{j+1}_p\|^2 \\
				\leq & \|\boldsymbol{e}^{n_0}_u\|^2 + \Delta t^2\|\nabla e^{n_0}_p\|^2 + \frac12(1+C_3)\Delta t^2\|\nabla e^{n_0}_p - \nabla e^{n_0-1}_p\|^2  + \nu\Delta t\|\nabla e^{n_0}_u\|^2 \\
				&+ CT(C_4+1)\sum_{i=r_u+1}^{d_u}\lambda_i\left(1 + \|\nabla\varphi_i\|^2\right) + CT(C_5+\gamma_p)\sum_{i=r_p+1}^{d_p}\theta_i\left(1 + \|\nabla\psi_i\|^2\right) \\
				&+CTC_6\left(h^{2k} + \Delta t^2\right)
			\end{aligned}
		\end{equation}
	where constants $C_i$, $i=3,\cdots,6$, are defined in (\ref{6ConsDefi}).
	\end{Lemma}	
    
    \begin{Remark}
    	We remark that the condition $\Delta t \leq \mathcal{O}(h^2)$ assumed in the proof of Lemma \ref{Lemma-Stab} which will be utilized soon, is not very difficult to meet in the convergence analysis of reduced-order solutions, since our ROM is constructed on the FE solutions, whose accuracy needs to be guaranteed by assuming the temporal scale $\Delta t$ and spatial scale $h$ to satisfy some certain relations, e.g., $\Delta t=\mathcal{O}(h^{k+1})$, $k$ is the polynomial degree of velocity FE space, and meanwhile $k \geq 1$ in those classical inf-sup stable FE pairs; in other words, the condition $\Delta t \leq \mathcal{O}(h^2)$ is actually being implied in FOM in the computation.  
    \end{Remark}
     
	\begin{proof}
	\noindent Subtracting POD-ROM (\ref{CDA-Proj-POD-ROM}) from FEM-FOM (\ref{FEM-FOM}) after taking $(\boldsymbol{v}_h,q_h) = (\boldsymbol{v}_r,q_r)$ to get
	\begin{equation}\label{ErroEqua-1}
	\begin{aligned}
		\frac{1}{\Delta t}(\boldsymbol{e}^{n+1}_u - \boldsymbol{e}^n_u,\boldsymbol{v}_r) &+ \nu(\nabla \boldsymbol{e}^{n+1}_u, \nabla \boldsymbol{v}_r) + [b(\boldsymbol{u}^n_h,\boldsymbol{u}^{n+1}_h,\boldsymbol{v}_r) - b(\boldsymbol{u}^n_r,\boldsymbol{u}^{n+1}_r,\boldsymbol{v}_r)] \\
		&- [(2p^n_h-p^{n-1}_h, \nabla\cdot \boldsymbol{v}_r) - (2p^n_r-p^{n-1}_r, \nabla\cdot \boldsymbol{v}_r)] + \gamma_u(I_H(\boldsymbol{u}^{n+1} - \boldsymbol{u}^{n+1}_r), I_H\boldsymbol{v}_r)= 0, \\
		(\nabla\cdot \boldsymbol{e}^{n+1}_u, q_r) &+ \Delta t(\nabla e^{n+1}_p - \nabla e^n_p, \nabla q_r) + \gamma_p(I_H(p^{n+1} - p^{n+1}_r), I_Hq_r) = 0.
	\end{aligned}
    \end{equation}		
	Letting $(\boldsymbol{v}_r, q_r) = (\boldsymbol{\xi}^{n+1}_u, 2\xi^n_p-\xi^{n-1}_p)$, adding and rearranging to get error equation
	\begin{equation}\label{ErroEqua-2}
	\begin{aligned}
		\frac{1}{\Delta t}(\boldsymbol{\xi}^{n+1}_u &- \boldsymbol{\xi}^n_u, \boldsymbol{\xi}^{n+1}_u) + \nu\|\nabla \boldsymbol{\xi}^{n+1}_u\|^2 + \Delta t(\nabla \xi^{n+1}_p - \nabla \xi^n_p, 2\nabla\xi^n_p-\nabla\xi^{n-1}_p) \\
		&+ \gamma(I_H(p^{n+1} - p^{n+1}_r), I_H(2\xi^n_p-\xi^{n-1}_p)) + \gamma_u(I_H(\boldsymbol{u}^{n+1} - \boldsymbol{u}^{n+1}_r), I_H\boldsymbol{\xi}^{n+1}_u)\\
		=& -\frac{1}{\Delta t}\left(\boldsymbol{\eta}^{n+1}_u - \boldsymbol{\eta}^n_u, \boldsymbol{\xi}^{n+1}_u\right) - \nu(\nabla\boldsymbol{\eta}^{n+1}_u, \nabla\boldsymbol{\xi}^{n+1}_u) - [b(\boldsymbol{u}^n_h,\boldsymbol{u}^{n+1}_h, \boldsymbol{\xi}^{n+1}_u) - b(\boldsymbol{u}^n_r,\boldsymbol{u}^{n+1}_r,\boldsymbol{\xi}^{n+1}_u)] \\
		&+ (2\eta^n_p-\eta^{n-1}_p, \nabla\cdot \boldsymbol{\xi}^{n+1}_u) - \Delta t(\nabla \eta^{n+1}_p - \nabla \eta^n_p, 2\nabla\xi^n_p-\nabla\xi^{n-1}_p) - (\nabla\cdot \boldsymbol{\eta}^{n+1}_u, 2\xi^n_p-\xi^{n-1}_p) \\
		:=& A_1 + A_2 + A_3 + A_4 + A_5 + A_6.
	\end{aligned}
    \end{equation}			
	$A_1$ vanishes since the $L^2$ projection orthogonality, and in the following we will use the inequalities $(a \pm b)^2 \leq 2(a^2 + b^2)$, $\forall a,b \in \mathbb{R}$ frequently. Firstly,
	$$
	\begin{aligned}
		|A_2| \leq \nu\|\nabla\boldsymbol{\eta}^{n+1}_u\|\cdot\|\nabla\boldsymbol{\xi}^{n+1}_u\| \leq \frac{1}{2}\nu\|\nabla\boldsymbol{\eta}^{n+1}_u\|^2 + \frac{1}{2}\nu\|\nabla\boldsymbol{\xi}^{n+1}_u\|^2.
	\end{aligned}
    $$
    Similarly,
    $$
    \begin{aligned}	
		|A_4| \leq 5\|2\nabla\eta^n_p-\nabla\eta^{n-1}_p\|^2 + \frac{1}{20}\|\boldsymbol{\xi}^{n+1}_u\|^2,
	\end{aligned}
    $$
    and
    $$
    \begin{aligned}	
		|A_6| &= |(\boldsymbol{\eta}^{n+1}_u, 2\nabla\xi^n_p-\nabla\xi^{n-1}_p)| \\
	    &= |(\boldsymbol{\eta}^{n+1}_u, (\nabla\xi^{n+1}_p-\nabla\xi^n_p)-(\nabla\xi^n_p-\nabla\xi^{n-1}_p)) - (\eta^{n+1}_u, \nabla\xi^{n+1}_p)| \\
	    &\leq |(\boldsymbol{\eta}^{n+1}_u, (\nabla\xi^{n+1}_p-\nabla\xi^n_p)-(\nabla\xi^n_p-\nabla\xi^{n-1}_p))| + |(\eta^{n+1}_u, \nabla\xi^{n+1}_p)| \\	
		&\leq 4\|\boldsymbol{\eta}^{n+1}_u\|^2 + \frac18(\|\nabla\xi^{n+1}_p-\nabla\xi^n_p\|^2 + \|\nabla\xi^n_p-\nabla\xi^{n-1}_p\|^2) + \frac{1}{2}\|\boldsymbol{\eta}^{n+1}_u\|^2 + \frac{1}{2}\|\nabla\xi^{n+1}_p\|^2.
    \end{aligned}
    $$		
	For $A_3$, we separate convection difference as
	$$
	\begin{aligned}
		|A_3| &= |b(\boldsymbol{u}^n_h,\boldsymbol{u}^{n+1}_h - \boldsymbol{u}^{n+1}_r, \boldsymbol{\xi}^{n+1}_u) - b(\boldsymbol{u}^n_h-\boldsymbol{u}^n_r,\boldsymbol{u}^{n+1}_r,\boldsymbol{\xi}^{n+1}_u)| \\
		&= |b(\boldsymbol{u}^n_h,\boldsymbol{\eta}^{n+1}_u, \boldsymbol{\xi}^{n+1}_u) + b(\boldsymbol{u}^n_h,\boldsymbol{\xi}^{n+1}_u, \boldsymbol{\xi}^{n+1}_u) + b(\boldsymbol{u}^n_h-\boldsymbol{u}^n_r,\boldsymbol{\xi}^{n+1}_u + \boldsymbol{u}^{n+1}_r, \boldsymbol{\xi}^{n+1}_u)| \\
		&\leq |b(\boldsymbol{u}^n_h,\boldsymbol{\eta}^{n+1}_u, \boldsymbol{\xi}^{n+1}_u)| + |b(\boldsymbol{\xi}^n_u,\Pi^u_r\boldsymbol{u}^{n+1}_h, \boldsymbol{\xi}^{n+1}_u)| + |b(\boldsymbol{\eta}^n_u,\Pi^u_r\boldsymbol{u}^{n+1}_h, \boldsymbol{\xi}^{n+1}_u)| \\
		&:= A_{31} + A_{32} + A_{33}.
    \end{aligned}
    $$			
	Therefore, using Lemma \ref{Stab-L2-Piu} to obtain
	$$
	\begin{aligned}
		A_{31} &\leq \|\boldsymbol{u}^n_h\|_{\infty}\|\nabla\boldsymbol{\eta}^{n+1}_u\|\|\boldsymbol{\xi}^{n+1}_u\| + \frac12\|\nabla\cdot \boldsymbol{u}^n_h\|_{L^{\infty}}\|\boldsymbol{\eta}^{n+1}_u\|\|\boldsymbol{\xi}^{n+1}_u\| \\
		&\leq 5\|\boldsymbol{u}^n_h\|^2_{\infty}\|\nabla\boldsymbol{\eta}^{n+1}_u\|^2 + \frac{1}{20}\|\boldsymbol{\xi}^{n+1}_u\|^2 + \frac54\|\nabla\cdot \boldsymbol{u}^n_h\|^2_{L^{\infty}}\|\boldsymbol{\eta}^{n+1}_u\|^2 + \frac{1}{20}\|\boldsymbol{\xi}^{n+1}_u\|^2 \\
		&\leq 5C^2_{L^{\infty}}\|\nabla\boldsymbol{\eta}^{n+1}_u\|^2 + \frac{1}{20}\|\boldsymbol{\xi}^{n+1}_u\|^2	+ \frac{5C^2_{W^{1,\infty}}}{4}\|\boldsymbol{\eta}^{n+1}_u\|^2 + \frac{1}{20}\|\boldsymbol{\xi}^{n+1}_u\|^2,
    \end{aligned}
    $$		
	where constants $C_{L^{\infty}}$, $C_{W^{1,\infty}}$are defined in (\ref{Cons-unh-Linfty}) and (\ref{Cons-unh-W1infty}). And by similar techniques and Lemma \ref{Lemma-Stab},
	$$
	\begin{aligned}
		A_{32} &\leq \|\nabla\Pi^u_r\boldsymbol{u}^{n+1}_h\|_{\infty}\|\boldsymbol{\xi}^{n}_u\|\|\boldsymbol{\xi}^{n+1}_u\| + \frac12\|\nabla\cdot \boldsymbol{\xi}^n_u\|\|\Pi^u_r\boldsymbol{u}^{n+1}_h\|_{\infty}\|\boldsymbol{\xi}^{n+1}_u\| \\
		&\leq \frac{C_{stab,1} C_{W^{1,\infty}}}{2}\|\boldsymbol{\xi}^n_u\|^2 + \frac{C_{stab,1} C_{W^{1,\infty}}}{2}\|\boldsymbol{\xi}^{n+1}_u\|^2 + \frac{\nu}{2}\|\nabla\boldsymbol{\xi}^n_u\|^2 + \frac{C^2_{stab,0}\cdot C^2_{L^{\infty}}}{8\nu}\|\boldsymbol{\xi}^{n+1}_u\|^2,
	\end{aligned}
	$$	
	where constants $C_{stab,0}$, $C_{stab,1}$ are defined in (\ref{Cons-Stab-0}) and (\ref{Cons-Stab-1}). Similarly, 
	$$
	\begin{aligned}
		A_{33} \leq 5C^2_{stab,1} C^2_{W^{1,\infty}}\|\boldsymbol{\eta}^n_u\|^2 + \frac{1}{20}\|\boldsymbol{\xi}^{n+1}_u\|^2 + \frac{5C^2_{stab,0} C^2_{L^{\infty}}}{4}\|\nabla\boldsymbol{\eta}^{n}_u\|^2 + \frac{1}{20}\|\nabla\boldsymbol{\xi}^{n+1}_u\|^2.
	\end{aligned}
	$$		
	Thus,
	$$
	\begin{aligned}		
		|A_3| &\leq A_{31} + A_{32} + A_{33} \\
		&\leq \frac12\left(\frac25 + C_{stab,1} C_{W^{1,\infty}} + \frac{C^2_{stab,0} C^2_{L^{\infty}}}{4\nu}\right)\|\boldsymbol{\xi}^{n+1}_u\|^2 + \frac{C_{stab,1} C_{W^{1,\infty}}}{2}\|\boldsymbol{\xi}^n_u\|^2 + \frac{\nu}{2}\|\nabla\boldsymbol{\xi}^n_u\|^2 \\
		&\quad+ \frac{5C^2_{W^{1,\infty}}}{4}\|\boldsymbol{\eta}^{n+1}_u\|^2 + 5C^2_{stab,1} C^2_{W^{1,\infty}}\|\boldsymbol{\eta}^n_u\|^2 + 5C^2_{L^{\infty}}\|\nabla\boldsymbol{\eta}^{n+1}_u\|^2 + \frac{5C^2_{stab,0} C^2_{L^{\infty}}}{4}\|\nabla\boldsymbol{\eta}^{n}_u\|^2.
    \end{aligned}
    $$		
	Moreover,
	$$
	\begin{aligned}		
		|A_5| &= |- \Delta t(\nabla \eta^{n+1}_p - \nabla \eta^n_p, \nabla\xi^{n+1}_p) + \Delta t(\nabla \eta^{n+1}_p - \nabla \eta^n_p, (\nabla\xi^{n+1}_p - \nabla\xi^{n}_p) - (\nabla\xi^{n}_p - \nabla\xi^{n-1}_p))| \\
		&\leq |\Delta t(\nabla \eta^{n+1}_p - \nabla \eta^n_p, \nabla\xi^{n+1}_p)| + |\Delta t(\nabla \eta^{n+1}_p - \nabla \eta^n_p, (\nabla\xi^{n+1}_p - \nabla\xi^{n}_p) - (\nabla\xi^{n}_p - \nabla\xi^{n-1}_p))|	\\
		&\leq \frac{\Delta t}{2}\|\nabla \eta^{n+1}_p - \nabla \eta^n_p\|^2 + \frac{\Delta t}{2}\|\nabla\xi^{n+1}_p\|^2 + 2\Delta t\|\nabla \eta^{n+1}_p - \nabla \eta^n_p\|^2 \\
		&\quad+ \frac14\Delta t\left(\|\nabla\xi^{n+1}_p - \nabla\xi^{n}_p\|^2 + \|\nabla\xi^{n}_p - \nabla\xi^{n-1}_p\|^2\right).
    \end{aligned}
    $$		
	For the last third term on LHS in error equation (\ref{ErroEqua-2}),
	$$
	\begin{aligned}
		\Delta t(\nabla\xi^{n+1}_p - \nabla\xi^{n}_p, 2\nabla\xi^{n}_p - \nabla\xi^{n-1}_p) &= \Delta t(\nabla\xi^{n+1}_p - \nabla\xi^{n}_p, \nabla\xi^{n+1}_p) \\
		&\quad- \Delta t(\nabla\xi^{n+1}_p - \nabla\xi^{n}_p, (\nabla\xi^{n+1}_p - \nabla\xi^{n}_p) - (\nabla\xi^{n}_p - \nabla\xi^{n-1}_p)).
    \end{aligned}
	$$	
	For the last term above, we use the second equation in (\ref{ErroEqua-1})
	$$
	(\nabla\cdot \boldsymbol{e}^{n+1}_u, q_r) + \Delta t(\nabla\xi^{n+1}_p - \nabla\xi^{n}_p, \nabla q_r) + \gamma_p(I_H(p^{n+1}-p^{n+1}_r), I_Hq_r) = -\Delta t(\nabla\eta^{n+1}_p - \nabla\eta^{n}_p, \nabla q_r),
	$$	
	to set $q_r = \xi^{n+1}_p - \xi^{n}_p$	 twice at $t=t_{n+1}$ and $t=t_{n}$ to have
	$$
	\begin{aligned}
		(\nabla\cdot \boldsymbol{e}^{n+1}_u, \xi^{n+1}_p - \xi^{n}_p) &+ \Delta t(\nabla\xi^{n+1}_p - \nabla\xi^{n}_p, \nabla \xi^{n+1}_p - \nabla\xi^{n}_p) + \gamma_p(I_H(p^{n+1}-p^{n+1}_r), I_H(\xi^{n+1}_p - \xi^{n}_p)) \\
		&= -\Delta t(\nabla\eta^{n+1}_p - \nabla\eta^{n}_p, \nabla \xi^{n+1}_p - \nabla\xi^{n}_p), \\
		(\nabla\cdot \boldsymbol{e}^{n}_u, \xi^{n+1}_p - \xi^{n}_p) &+ \Delta t(\nabla\xi^{n}_p - \nabla\xi^{n-1}_p, \nabla \xi^{n+1}_p - \nabla\xi^{n}_p) + \gamma_p(I_H(p^{n}-p^{n}_r), I_H(\xi^{n+1}_p - \xi^{n}_p)) \\
		&= -\Delta t(\nabla\eta^{n}_p - \nabla\eta^{n-1}_p, \nabla \xi^{n+1}_p - \nabla\xi^{n}_p),
    \end{aligned}
    $$		
	that is,	
	$$
	\begin{aligned}	
		&\Delta t(\nabla\xi^{n+1}_p - \nabla\xi^{n}_p, (\nabla\xi^{n+1}_p - \nabla\xi^{n}_p) - (\nabla\xi^{n}_p - \nabla\xi^{n-1}_p)) \\
		&= -(\nabla\cdot (\boldsymbol{e}^{n+1}_u - \boldsymbol{e}^{n}_u), \xi^{n+1}_p - \xi^{n}_p) \\
		&\quad- \gamma_p(I_H(p^{n+1}-p^{n+1}_r) - I_H(p^{n}-p^{n}_r), I_H(\xi^{n+1}_p - \xi^{n}_p)) \\
		&\quad- \Delta t((\nabla\eta^{n+1}_p - \nabla\eta^{n}_p) - (\nabla\eta^{n}_p - \nabla\eta^{n-1}_p), \nabla \xi^{n+1}_p - \nabla\xi^{n}_p) \\
		&:= B_1 + B_2 + B_3
    \end{aligned}
    $$	
	Then, 	
	$$
	\begin{aligned}		
		B_1 = (\boldsymbol{e}^{n+1}_u - \boldsymbol{e}^{n}_u, \nabla \xi^{n+1}_p - \nabla \xi^{n}_p) 
		\leq 2\left(2\|\boldsymbol{\eta}^{n+1}_u\|^2 + 2\|\boldsymbol{\eta}^{n}_u\|^2 + \|\boldsymbol{\xi}^{n+1}_u - \boldsymbol{\xi}^{n}_u\|^2\right) + \frac{1}{4}\|\nabla \xi^{n+1}_p - \nabla \xi^{n}_p\|,
	\end{aligned}
    $$
    and
    $$
    \begin{aligned}	
		B_2 &= -\gamma_p(I_H(p^{n+1}-\Pi^p_rp^{n+1}_h) - I_H(p^{n}-\Pi^p_rp^{n}_h), I_H(\xi^{n+1}_p - \xi^{n}_p)) - \gamma_p\|I_H(\xi^{n+1}_p - \xi^{n}_p)\|^2 \\
		&\leq \gamma_p\|I_H(p^{n+1}-\Pi^p_rp^{n+1}_h) - I_H(p^{n}-\Pi^p_rp^{n}_h)\|\cdot\|I_H(\xi^{n+1}_p - \xi^{n}_p)\| - \gamma_p\|I_H(\xi^{n+1}_p - \xi^{n}_p)\|^2 \\
		&\leq \frac{\gamma_p}{2}\|I_H(p^{n+1}-\Pi^p_rp^{n+1}_h) - I_H(p^{n}-\Pi^p_rp^{n}_h)\|^2 + \frac{\gamma_p}{2}\|I_H(\xi^{n+1}_p - \xi^{n}_p)\|^2 - \gamma_p\|I_H(\xi^{n+1}_p - \xi^{n}_p)\|^2 \\
		&\leq 2\gamma_p\left(\|I_H(p^{n+1} - p^{n+1}_h)\|^2 + \|I_H\eta^{n+1}_p\|^2 + \|I_H(p^{n} - p^{n}_h)\|^2 + \|I_H\eta^n_p\|^2\right) \\
		&\quad+ \frac{\gamma_p}{2}\|I_H(\xi^{n+1}_p - \xi^{n}_p)\|^2 - \gamma_p\|I_H(\xi^{n+1}_p - \xi^{n}_p)\|^2,
	\end{aligned}
    $$
    and
    $$
    \begin{aligned}	
		B_3 &\leq \Delta t\|(\nabla\eta^{n+1}_p - \nabla\eta^{n}_p) - (\nabla\eta^{n}_p - \nabla\eta^{n-1}_p)\|\cdot\|\nabla \xi^{n+1}_p - \nabla\xi^{n}_p\| \\
		&\leq \frac{\Delta t}{\epsilon_{14}}\left(\|\nabla\eta^{n+1}_p - \nabla\eta^{n}_p\|^2 + \|\nabla\eta^{n}_p - \nabla\eta^{n-1}_p\|^2\right) + \frac{\epsilon_{14}\Delta t}{2}\|\nabla \xi^{n+1}_p - \nabla\xi^{n}_p\|^2 \\
		&\leq 8\Delta t\left(\|\nabla\eta^{n+1}_p\|^2 + 2\|\nabla\eta^{n}_p\|^2 + \|\nabla\eta^{n-1}_p\|^2\right) + \frac{\Delta t}{8}\|\nabla \xi^{n+1}_p - \nabla\xi^{n}_p\|^2.
    \end{aligned}
    $$	
	Moreover, for the last second term on LHS in (\ref{ErroEqua-2}),
	$$
	\begin{aligned}
		\gamma_p(I_H(p^{n+1} - p^{n+1}_r), I_H(2\xi^n_p-\xi^{n-1}_p))	&= \gamma_p(I_H(p^{n+1} - p^{n+1}_r), I_H\xi^{n+1}_p) \\
		&\quad- \gamma_p(I_H(p^{n+1} - p^{n+1}_r), I_H(\xi^{n+1}_p-\xi^{n}_p) - I_H(\xi^n_p-\xi^{n-1}_p)) \\
		&:= B_4 - B_5,
	\end{aligned}
    $$	
	and for $B_4$, we estimate it similarly as done in \cite{CDA-2019-CMAME} to obtain
	$$
	\begin{aligned}
		B_4 = \gamma_p\|\xi^{n+1}_p\|^2 + \gamma_p(I_H(p^{n+1} - \Pi^p_rp^{n+1}_h), I_H\xi^{n+1}_p) + \gamma_p\|I_H\xi^{n+1}_p - \xi^{n+1}_p\|^2 + 2\gamma_p(\xi^{n+1}_p, I_H\xi^{n+1}_p - \xi^{n+1}_p),
	\end{aligned}
	$$	
	where we can get
	$$
	\begin{aligned}
		-\gamma_p(I_H(p^{n+1} - \Pi^p_rp^{n+1}_h), I_H\xi^{n+1}_p) &\leq \gamma_p\|I_H(p^{n+1} - \Pi^p_rp^{n+1}_h)\|\cdot\|I_H\xi^{n+1}_p\| \\
		&\leq 8C^2_{2,H}\gamma_p\left(\|I_H(p^{n+1} - p^{n+1}_h)\|^2 + \|I_H\eta^{n+1}_p\|^2\right) + \frac{\gamma_p}{8}\|\xi^{n+1}_p\|^2
	\end{aligned}
    $$
    and
    $$
    \begin{aligned}
    	-2\gamma_p(\xi^{n+1}_p, I_H\xi^{n+1}_p - \xi^{n+1}_p) &\leq \gamma_p\|\xi^{n+1}_p\|^2 + \gamma_p\|I_H\xi^{n+1}_p - \xi^{n+1}_p\|^2 \\
    	&\leq 2\gamma_p C^2_{P}\|\nabla\xi^{n+1}_p - \nabla\xi^{n}_p\|^2 + 2\gamma_p C^2_{P}\|\nabla\xi^{n}_p\|^2 + \gamma_p\|I_H\xi^{n+1}_p - \xi^{n+1}_p\|^2.
    \end{aligned}
    $$
	For $B_5$,
	$$
	\begin{aligned}
		B_5 &\leq \gamma_p\|I_H(p^{n+1} - p^{n+1}_r)\|\cdot\|I_H(\xi^{n+1}_p-\xi^{n}_p) - I_H(\xi^n_p-\xi^{n-1}_p)\| \\
		&\leq \frac{\gamma_p}{2C^2_{2,H}}\|I_H(p^{n+1} - p^{n+1}_r)\|^2 + \frac{C^2_{2,H}\gamma_p}{2}\|I_H(\xi^{n+1}_p-\xi^{n}_p) - I_H(\xi^n_p-\xi^{n-1}_p)\|^2 \\
		&\leq \frac{\gamma_p}{2C^2_{2,H}}\left[3\|I_H(p^{n+1} - \Pi^p_rp^{n+1}_h)\|^2 + \frac32\|I_H\xi^{n+1}_p\|^2\right] \\
		&\quad+ \frac{C^2_{2,H}\gamma_p}{2}\left(\|I_H(\xi^{n+1}_p-\xi^{n}_p)\|^2 + \|I_H(\xi^n_p-\xi^{n-1}_p)\|^2\right)\\
		&\leq \frac{3\gamma_p}{2C^2_{2,H}}\left[\|I_H(p^{n+1} - p^{n+1}_h)\|^2 + \|I_H\eta^{n+1}_p\|^2 + \frac12C^2_{2,H}\|\xi^{n+1}_p\|^2\right] \\
		&\quad+ \frac{\gamma_p C^4_{2,H}C^2_{P}}{2}\left(\|\nabla \xi^{n+1}_p-\nabla \xi^{n}_p\|^2 + \|\nabla \xi^n_p-\nabla \xi^{n-1}_p\|^2\right).
	\end{aligned}
	$$	
	Finally, for the last term on LHS in (\ref{ErroEqua-2}), just like done for bounding $B_4$,
	$$
	\begin{aligned}
		\gamma_u(I_H(\boldsymbol{u}^{n+1} - \boldsymbol{u}^{n+1}_r), I_H\boldsymbol{\xi}^{n+1}_u) &= \gamma_u\|\boldsymbol{\xi}^{n+1}_u\|^2 + 2\gamma_u(\boldsymbol{\xi}^{n+1}_u, I_H\boldsymbol{\xi}^{n+1}_u-\boldsymbol{\xi}^{n+1}_u) \\
		&\quad+ \gamma_u(I_H\boldsymbol{\eta}^{n+1}_u, I_H\boldsymbol{\xi}^{n+1}_u) + \gamma_u\|I_H\boldsymbol{\xi}^{n+1}_u-\boldsymbol{\xi}^{n+1}_u\|^2,
    \end{aligned}
    $$
    then
    $$
    \begin{aligned}
    	- \gamma_u(I_H\boldsymbol{\eta}^{n+1}_u, I_H\boldsymbol{\xi}^{n+1}_u) & \leq \frac{\gamma_u}{2}C^2_{2,H}\|\boldsymbol{\eta}^{n+1}_u\|^2 + \frac{\gamma_u}{2}C^2_{2,H}\|\boldsymbol{\xi}^{n+1}_u\|^2, \\
    	2\gamma_u(\boldsymbol{\xi}^{n+1}_u, I_H\boldsymbol{\xi}^{n+1}_u-\boldsymbol{\xi}^{n+1}_u) &\leq \frac{\gamma_u}{2}\|\boldsymbol{\xi}^{n+1}_u\|^2 + C^2_{1,H}H^2\gamma_u\|\nabla\boldsymbol{\xi}^{n+1}_u\|^2 + \gamma_u\|I_H\boldsymbol{\xi}^{n+1}_u-\boldsymbol{\xi}^{n+1}_u\|^2.
    \end{aligned}
    $$
	Collecting all results above, using the projection error estimates in Assumption \ref{Assu}	and rearranging to get
    $$
    \begin{aligned}
    	\frac{1}{2\Delta t}\bigg(\|\boldsymbol{\xi}^{n+1}_u\|^2 &- \|\boldsymbol{\xi}^n_u\|^2\bigg) + \frac{1}{2\Delta t}\|\boldsymbol{\xi}^{n+1}_u - \boldsymbol{\xi}^{n}_u\|^2 + \frac12\nu\bigg(\|\nabla\boldsymbol{\xi}^{n+1}_u\|^2 - \|\nabla\boldsymbol{\xi}^n_u\|^2\bigg) + \frac12\Delta t\bigg(\|\nabla \xi^{n+1}_p\|^2 - \|\nabla \xi^{n}_p\|^2\bigg) \\
    	& + \frac14\Delta t\bigg(\|\nabla \xi^{n+1}_p - \nabla\xi^{n}_p\|^2 - \|\nabla \xi^{n}_p - \nabla\xi^{n-1}_p\|^2\bigg) + \frac18\gamma_p\|\xi^{n+1}_p\|^2 + \frac12\gamma_p\|I_H(\xi^{n+1}_p - \xi^{n}_p)\|^2 \\
    	\leq& \frac12\left[\frac92 + C_{stab,1}\cdot C_{W^{1,\infty}} + \frac{C^2_{stab,0}\cdot C^2_{L^{\infty}}}{4\nu} + \frac12\gamma_u(C^2_{2,H} + 1)\right]\|\boldsymbol{\xi}^{n+1}_u\|^2 \\
    	&+ \frac12\bigg(C_{stab,1}C_{W^{1,\infty}} + 8\bigg)\|\boldsymbol{\xi}^n_u\|^2 + \frac12\bigg(\Delta t + 1\bigg)\|\nabla \xi^{n+1}_p\|^2 + 2\gamma_p C^2_{P}\|\nabla \xi^{n}_p\|^2\\
    	& + \left(\frac38 + \frac12\Delta t + \frac12\gamma_p C^4_{2,H}C^2_{P} + 2\gamma_p C^2_{P}\right)\|\nabla \xi^{n+1}_p - \nabla\xi^{n}_p\|^2\\
    	& + \left(\frac18 + \frac12\gamma_p C^4_{2,H}C^2_{P}\right)\|\nabla\xi^{n}_p - \nabla\xi^{n-1}_p\|^2 \\
    	&+ C\sum_{i=r_u+1}^{d_u}\lambda_i\left[\left(\frac{25}{2} + \frac{5}{4}C^2_{W^{1,\infty}} + 5C^2_{stab,1}\cdot C^2_{W^{1,\infty}} + \frac12\gamma_u\right)\right. \\
    	&\quad\left.+ \left(\frac{\nu}{2} + 5C^2_{L^{\infty}} + \frac54C^2_{stab,0}\cdot C^2_{L^{\infty}}\right)\cdot\|\nabla\varphi_i\|^2\right] \\
    	&+ C\sum_{i=r_p+1}^{d_p}\theta_i\left[\left(4\gamma_p C^2_{2,H} + 8C^4_{2,H}\gamma_p + \frac32\gamma_p \right) + \bigg(40 + 42\Delta t\bigg)\|\nabla\psi_i\|^2\right] \\
    	&+CC^2_{2,H}\left(4\gamma_p + 8C^2_{2,H}\gamma_p + \frac{3\gamma_p}{2C^2_{2,H}}\right)\bigg(h^{2k} + \Delta t^2\bigg)
    \end{aligned}
    $$
	Define
	\begin{equation}\label{6ConsDefi}
	\begin{aligned}	
		C_1 :=& \left[\frac{25}2 + 2C_{stab,1}\cdot C_{W^{1,\infty}} + \frac{C^2_{stab,0}\cdot C^2_{L^{\infty}}}{4\nu} + \frac12\gamma_u\bigg(C^2_{2,H} + 1\bigg)\right], \\
		C_2 :=& \Delta t^{-2}\bigg(\Delta t + 1 + 4\gamma_p C^2_{P}\bigg), \\
		C_3 :=& 2\Delta t^{-2}\left(1 + \Delta t + 2\gamma_p C^4_{2,H}C^2_{P} + 4\gamma_p C^2_{P}\right), \\
		C_4 :=& \max\left\{\left(\frac{25}{2} + \frac{5}{4}C^2_{W^{1,\infty}} + 5C^2_{stab,1}\cdot C^2_{W^{1,\infty}} + \frac12\gamma_u\right),\;\left(\frac{\nu}{2} + 5C^2_{L^{\infty}} + \frac54C^2_{stab,0}\cdot C^2_{L^{\infty}}\right)\right\}, \\
		C_5 :=& \max\left\{\left(4\gamma_p C^2_{2,H} + 8C^4_{2,H}\gamma_p + \frac32\gamma_p \right),\;\bigg(40 + 42\Delta t\bigg)\right\}, \\
		C_6 :=& C^2_{2,H}\left(4\gamma_p + 8C^2_{2,H}\gamma_p + \frac{3\gamma_p}{2C^2_{2,H}}\right).
	\end{aligned}		
    \end{equation}
	Multiplying	$2\Delta t$ and adding from $n_0$ to $n$, we get
	$$
	\begin{aligned}	
		\bigg(\|\boldsymbol{\xi}^{n+1}_u\|^2 &+ \Delta t^2\|\nabla \xi^{n+1}_p\|^2 + \frac12\Delta t^2\|\nabla\xi^{n+1}_p - \nabla\xi^{n}_p\|^2\bigg) + \nu\Delta t\|\nabla\boldsymbol{\xi}^{n+1}_u\|^2 \\
		&+ \gamma_p\Delta t\sum_{j=n_0}^{n}\bigg(\frac14\|\xi^{j+1}_p\|^2 + \|I_H(\xi^{j+1}_p - \xi^{j}_p)\|^2\bigg) \\
		\leq & \bigg(\|\boldsymbol{\xi}^{n_0}_u\|^2 + \Delta t^2\|\nabla \xi^{n_0}_p\|^2 + \frac12\Delta t^2\|\nabla\xi^{n_0}_p - \nabla\xi^{n_0-1}_p\|^2\bigg) + \frac12C_3\Delta t^2\|\nabla\xi^{n_0}_p - \nabla\xi^{n_0-1}_p\|^2  + \nu\Delta t\|\nabla\boldsymbol{\xi}^{n_0}_u\|^2 \\
		&+ \Delta t\sum_{j=n_0}^{n}\bigg(C_1\|\xi^{j+1}_u\|^2 + C_2\Delta t^2\|\nabla\xi^{j+1}_p\|^2 + \frac12C_3\Delta t^2\|\nabla\xi^{j+1}_p - \nabla\xi^{j}_p\|^2\bigg) \\
		&+ CTC_4\sum_{i=r_u+1}^{d_u}\lambda_i\left(1 + \|\nabla\varphi_i\|^2\right) + CTC_5\sum_{i=r_p+1}^{d_p}\theta_i\left(1 + \|\nabla\psi_i\|^2\right) +CTC_6\left(h^{2k} + \Delta t^2\right)
	\end{aligned}$$		
	Setting $\tilde{C}:=\max\{C_1,C_2,C_3\}$, and choosing $\Delta t$ such that $\tilde{C}\Delta t \leq 1/2$, then using Lemma \ref{Gronwall}, Assumption \ref{Assu} and dropping unneeded terms, and finally we reach the final result based on the fact that our initial condition in ROM are obtained by projecting high-fidelity solution, i.e., the initial approximation errors $\xi^{n_0}_u=\xi^{n_0}_p=\xi^{n_0-1}_p=0$.		
	\end{proof}		
	\indent Finally, by triangle inequality, Lemma \ref{Conv-Exact-FE} and Lemma \ref{Conv-FE-POD}, we obtain the following theorem which states the convergence between the continuous solutions and reduced-order solutions.  		
	\begin{Theorem}\label{Conv-Exact-POD}
		(Error estimate for CDA-Proj-POD-ROM) For any $n \geq n_0$($n_0 \geq 1$), let $\left(\boldsymbol{u}^n, p^n\right)$ is the solution of (\ref{ContVariNSEs}) at $t=t_n$,  $\left(\boldsymbol{u}_r^n, p_r^n\right)$ denotes the POD-based reduced-order solutions obtained in (\ref{CDA-Proj-POD-ROM}), then we assume $\Delta t = \mathcal{O}(h^2)$ and there exists a constant $C$, such that the following convergence estimate holds:
	\begin{equation}
	\begin{aligned}
		\left\|\boldsymbol{u}^n-\boldsymbol{u}_r^n\right\|^2 & + \nu \Delta t \left\|\nabla\left(\boldsymbol{u}^n-\boldsymbol{u}_r^n\right)\right\|^2 + \Delta t^2 \left\|\nabla\left(p^n-p_r^n\right)\right\|^2 + \gamma_p\|p-p_r\|^2_{\ell(L^2)} \\
		\leq&  C\left[\sum_{i=r+1}^{d_{\tilde{\boldsymbol{u}}}} \lambda_i\bigg(1 + \|\nabla\varphi_i\|^2\bigg) + \sum_{i=r+1}^{d_p} \theta_i\bigg(1 + \|\nabla\psi_i\|^2\bigg) + \Delta t^2 + h^{2k}\right]
	\end{aligned}
    \end{equation}		
    \end{Theorem}			
    \newpage
    \section{Numerical Test}\label{section-4}
	In this section, we will present some numerical results for the CDA-Proj-POD-ROM(\ref{CDA-Proj-POD-ROM}) proposed and analyzed before, for the following two purposes: 1) verifying convergence about the number of POD modes; 2) the necessity of adding DA terms to obtain stable and accurate reduced-order solutions. The numerical experiments are performed on the famous benchmark problem of the 2D unsteady flow past a cylinder with circular cross-section at $Re=100$ \cite{Benchmark-1996}. This problem was also used in the ROM settings with stable reduced-order pressure in NSE in \cite{POD-LPS-SINUM-2021,POD-LPS-SINUM-2020,Post-proc-POD-ROM-PresRecove-2022-CMAME} etc. The open-source FE package iFEM \cite{iFEM} has been used to run the numerical experiments.  \\
	\indent Following \cite{Benchmark-1996}, the computational domain is a rectangular channel with a circular hole(see Fig. \ref{CompDomain_TwoMeshes_ObseLoca} on the top):
	$$
	\Omega=\{(0,2.2) \times(0,0.41)\} \backslash\left\{\boldsymbol{x}:(\boldsymbol{x}-(0.2,0.2))^2 \leq 0.05^2\right\}.
	$$
	\indent The parabolic inflow and outflow profile 
	$$
	\boldsymbol{u}(0, y)=\boldsymbol{u}(2.2, y)=0.41^{-2}(1.2 y(0.41-y), 0), \quad 0 \leq y \leq 0.41,
	$$
	is prescribed, no-slip conditions are imposed at the other boundaries. Using the mean values of the inflow velocity $\bar{U}=1\,\text{m/s}$ and the diameter of the cylinder $D=0.1\,\text{m}$, the $Re$ considered is $Re=100$. The density is given by $\rho = 1\,\text{kg/m}^3$ and no external forcing, i.e. $\boldsymbol{f}=0\,\text{m/s}^2$. In the fully developed periodic regime, a vortex shedding is formed behind the obstacle, which is the well-known von Kármán vortex street(see Fig. \ref{Magnitude_of_FE_and_POD_velocity_and_pressure}). \\
	\indent The benchmark parameters are the drag and lift coefficients at the cylinder, and in order to reduce the negative effect of line integral on the boundary of the cylinder, we adopt the volume integral defined in \cite{benchmark-John-2001-IJNMF} as
	$$
	\begin{aligned}
		c_D & =-\frac{2}{D \bar{U}^2}\left[\nu\left(\nabla \boldsymbol{u}, \nabla \boldsymbol{v}_D\right)+b\left(\boldsymbol{u}, \boldsymbol{u}, \boldsymbol{v}_D\right)-\left(p, \nabla \cdot \boldsymbol{v}_D\right)\right], \\
		c_L & =-\frac{2}{D \bar{U}^2}\left[\nu\left(\nabla \boldsymbol{u}, \nabla \boldsymbol{v}_L\right)+b\left(\boldsymbol{u}, \boldsymbol{u}, \boldsymbol{v}_L\right)-\left(p, \nabla \cdot \boldsymbol{v}_L\right)\right],
	\end{aligned}
	$$		
	for arbitrary functions $\boldsymbol{v}_D \in \boldsymbol{H}^1$(resp. $\boldsymbol{v}_L \in \boldsymbol{H}^1$) such that $\boldsymbol{v}_D = (1, 0)^T$(resp. $\boldsymbol{v}_L = (0, 1)^T$) on the boundary of the cylinder and vanishes on the other boundaries. Another relevant quantity of interest is the kinetic energy of the flow, given by
	$$
	E_{K i n}=\frac{1}{2}\|\boldsymbol{u}\|_{\mathbf{L}^2}^2.
	$$
	\begin{figure}
		\centering
		\includegraphics[width=16.0005cm,height=7.95cm]{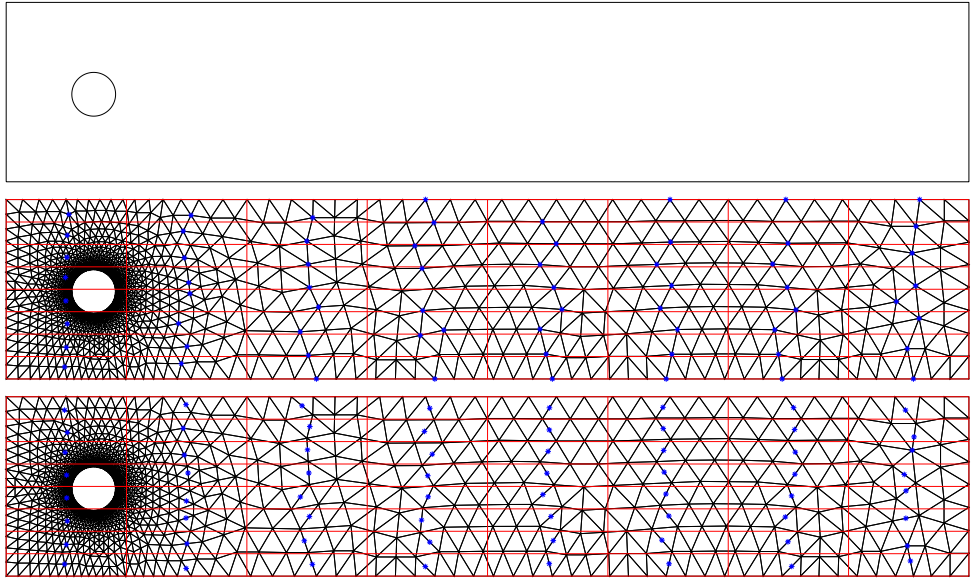}
		\caption{Computational domain(top) and two meshes(fine in black and coarse in red) with observation locations on $P^1$ FE nodes(middle) and $P^2$ FE nodes(bottom) both labeled in blue star.[Colors figure can be viewed at the website version]}\label{CompDomain_TwoMeshes_ObseLoca}
	\end{figure}
    \begin{figure}
    	\centering
    	\includegraphics[width=16.0005cm,height=7.95cm]{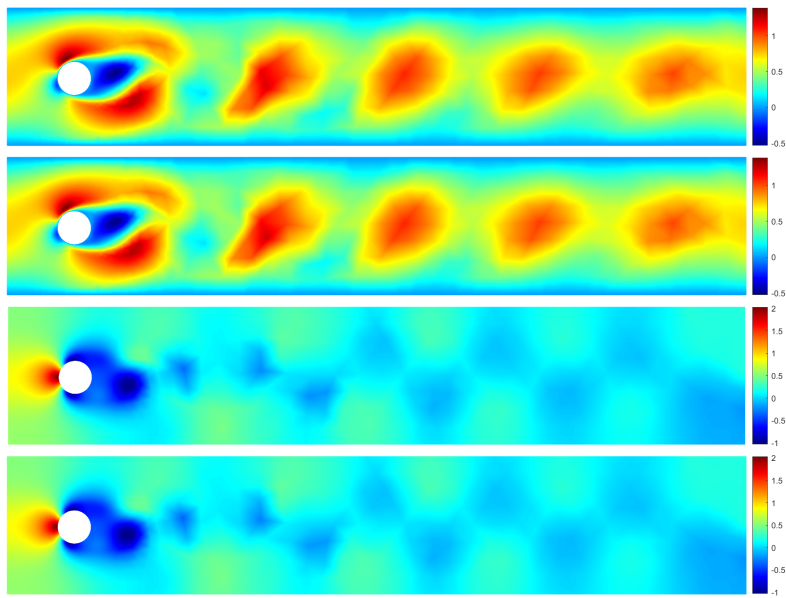}
    	\caption{Final temporal velocity(first two sub-figures) and pressure(last two sub-figures) magnitude; from top to bottom on two sub-figures are FE high-fidelity solutions(top) and reduced-order solutions(bottom).[Colors figure can be viewed at the website version]}\label{Magnitude_of_FE_and_POD_velocity_and_pressure}
    \end{figure}
    \begin{table}[htb!]
    	\centering
    	\caption{Quantitative comparisons of maximal drag coefficient $c_{\text{d},\max}$, maximal lift coefficient $c_{\text{l},\max}$ and pressure difference at final time $T=8$ with DNS, G-ROM(i.e., $\gamma_u = \gamma_p = 0$) and CDA-G-ROM(i.e., $\gamma_u=\gamma_p=100$).}\label{Comp-Table}
    	\begin{tabular}{ccccc}
    		\toprule[1pt]
    		& \quad DNS \quad & \quad G-ROM \quad & \quad CDA-G-ROM \quad & \quad Ref.\cite{Benchmark-1996}  \\
    		\midrule[0.5pt]
    		\quad  $c_{\text{d},\max}$ \quad & \quad 3.2247 \quad & \quad 3.2380 \quad & \quad 3.2226 \quad & \quad [3.22,$\;$3.24]  \\  
    		\quad  $c_{\text{l},\max}$ \quad & \quad 1.0007 \quad & \quad 0.9902 \quad & \quad 1.0032 \quad & \quad [0.99,$\;$1.01]  \\  
    		\quad  $\Delta_{p}(8\;\text{s})$ \quad & \quad 2.4357 \quad & \quad 2.3409 \quad & \quad 2.4359 \quad & \quad [2.46,$\;$2.50]  \\  
    		\bottomrule[1pt]
    	\end{tabular}
    \end{table}
	\subsection{DNS, CDA and POD modes}
	\indent The DNS to compute the snapshots is FEM-FOM (\ref{FEM-FOM}), with a spatial discretization using the mixed inf-sup stable $\boldsymbol{P}^2-P^1$ Taylor-Hood elements for velocity-pressure on a uniform refined mesh that provides 13496 velocity d.o.fs. and 1754 pressure d.o.fs., $\Delta t=1e-4$ and the whole time integration is performed on $[0,8]$. The DNS simulation is started from rest for both velocity and pressure, i.e., the initial conditions are zeros velocity field and pressure field. As stated in \cite{RefValue-John-2004-IJNMF} that the first-order discretization yields much more inaccurate results; moreover, fully resolved computation of the NSEs in this benchmark problem needs more than  100K d.o.fs., which mean the simulations here are all under-resolved, so we have to note that we could not expect exact agreement with solutions of our scheme with the reference ones; however, we still do expect the answers to be close.  \\
	\indent For the DA computation, we choose $I_H$ to be an approximation of the $L^2$ projection operator onto piece-wise constant on a coarse mesh, which was proposed and analyzed about its stability and convergence in \cite{CDA-2021-NMPDE}. The coarse mesh for DA is constructed using the intersection of uniform rectangular meshes with diameter $H$ denoting the width of each rectangular. For the observation locations' selection on each coarse element, we choose the node that is lying on the fine FE nodes within that coarse element, and meanwhile is closest to the center of this coarse element, to be the observation location on the coarse mesh to obtain the observation data and thus implementing our CDA-Proj-POD-ROM(\ref{CDA-Proj-POD-ROM}).	Fig.\ref{CompDomain_TwoMeshes_ObseLoca}(middle) shows in black the fine mesh for pressure and in red the coarse mesh with $H=2.2/8$(resulting in 64 DA points), observation points are labeled in the blue star on $P^1$ FE nodes; The sub-figure on Fig. \ref{CompDomain_TwoMeshes_ObseLoca}(bottom) corresponds the velocity, with the same number of observations points but on $P^2$ FE mesh. Fig. \ref{Magnitude_of_FE_and_POD_velocity_and_pressure} illustrates velocity and pressure magnitude; specifically, from top to bottom, the four sub-figures correspond in turn to FE high-fidelity velocity(first), POD reduced-order velocity(second), FE high-fidelity pressure(third) and finally POD reduced-order pressure(last). Tab. \ref{Comp-Table} also shows a quantitative comparison of quantities of interest about DNS, G-ROM(i.e., $\gamma_u = \gamma_p = 0$), and CDA-G-ROM(i.e., $\gamma_u=\gamma_p=100$) both with $r=19$, nudging velocity and pressure enhances the accuracy. \\
	\indent The POD velocity and pressure modes are generated simultaneously in $L^2$ inner product by the method of snapshots by storing continuously every DNS velocity-pressure solution from $t=5$ to $t=5.4$, during which the flow has reached a periodic-in-time state and the length of the time interval is greater than the length of one flow period. In other words, we run FOM on $[5,5.4]$ and then consider ROM to predict/extrapolation in time on $[5.4,8]$ to test the stability and accuracy of the reduced-order velocity and pressure, the latter time interval is six times larger than the one used to compute the snapshots and generate POD modes. Fig. \ref{EigeValuDesc-VeloPres} shows the decay of POD velocity ($\lambda_{i}$, $i=1,\cdots,71=d_u$) and pressure ($\theta_i$, $i=1,\cdots,75=d_p$) eigenvalues. We denote the corresponding captured system's energy as $100\sum_{i=1}^{r_u}\lambda_{i}/\sum_{i=1}^{d_u}\lambda_{i}$ for velocity, $100\sum_{i=1}^{r_p}\theta_{i}/\sum_{i=1}^{d_p}\theta_{i}$ for pressure, and we adopt $r_u=r_p=8,19$ within the forthcoming ROM computation, and those POD modes capture more than $99.95\%,99.9999\%$ of the system's velocity-pressure energy. We also remark that shorter time interval for taking snapshots is allowed, especially with the help of DA; that is, inaccurate snapshots(i.e., less than a flow period) can also be used to obtain relatively accurate and stable reduced-order solutions through DA technique. For more details, we refer to \cite{CDA-2019-CMAME,CDA-POD-2022-JCAM}.
	\begin{figure}
		\centering
		\includegraphics[width=14.9338cm,height=7.4200cm]{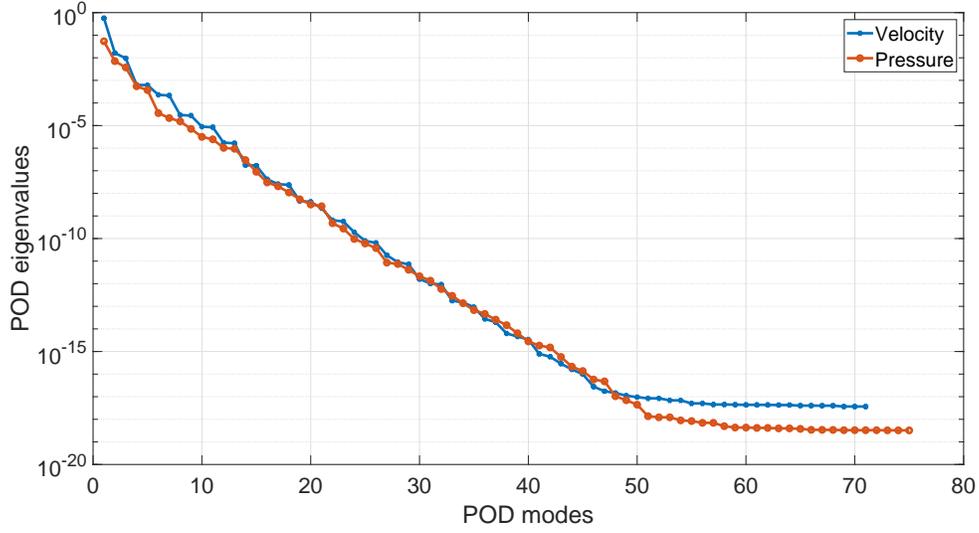}
		\caption{POD velocity and pressure eigenvalues}\label{EigeValuDesc-VeloPres}
	\end{figure}
    \begin{figure}
    	\centering
    	\includegraphics[width=14.9338cm,height=7.4200cm]{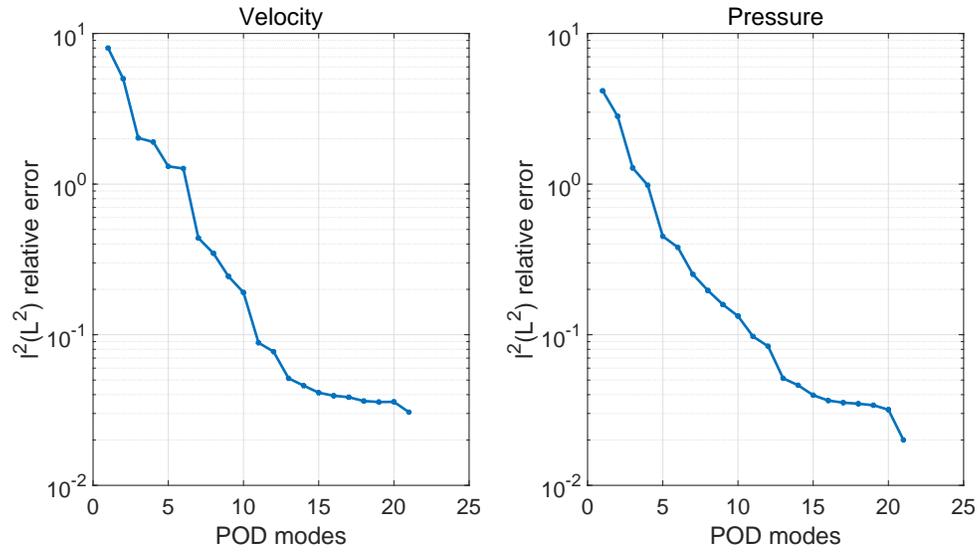}
    	\caption{Convergence of error for velocity(left) and pressure(right) concerning POD modes with $\gamma_u=100$ and $\gamma_p=100$.}\label{Conv-DiscTemp}
    \end{figure}
	\subsection{Stability, Convergence and High-Accuracy}
	\indent We first in this subsection verify numerically the theoretical results obtained from the previous sections, i.e. stability and convergence, especially the {\itshape a prior} pressure stability. We firstly verify the result of Theorem \ref{Conv-Exact-POD}, which states that the numerical error of reduced-order solutions will diminish with the increase of POD modes, up to the temporal and spatial dicretization error. Fig.\ref{Conv-DiscTemp} shows that the relative $L^2$ error for both velocity and pressure in discrete time integral decreases as increasing number of POD modes. Fig.\ref{RelaErr-nDA8-GammaP_0_100-GammaU_0_100-r-19} and Fig.\ref{3ParasEnergy-nDA8-gammaU-0-100-gammaP-0-100-r19} illustrate the necessity of CDA terms, in terms of relative error(Fig.\ref{RelaErr-nDA8-GammaP_0_100-GammaU_0_100-r-19}) and pressure difference at final time, lift and drag coefficients and the kinetic energy(Fig.\ref{3ParasEnergy-nDA8-gammaU-0-100-gammaP-0-100-r19}) using $r=19$ separately. In particular, we can see from Fig.\ref{RelaErr-nDA8-GammaP_0_100-GammaU_0_100-r-19} that without CDA(i.e., $\gamma_u=\gamma_p=0$), the relative error of velocity increases monotonically and continuously; at the same time, if only the velocity nudging term is added, the pressure's relative error fluctuates dramatically. Also we find from the temporal evolution concerning velocity that the performance of adding pressure nudging term solely(i.e., $\gamma_u=0,\;\gamma_p=100$) is not as good as the one of adding both(i.e., $\gamma_u=\gamma_p=100$). A similar phenomenon is observed in Fig.\ref{3ParasEnergy-nDA8-gammaU-0-100-gammaP-0-100-r19}, where the presence of velocity penalty term(i.e., $\gamma_u\neq 0$) makes the kinetic energy of reduced-order velocity more accurate.  \\
	\indent In order to further illustrate that CDA has the effects of stabilizing under-resolved reduced-order modeling, we need first to determine the optimal value of $r$, i.e., the smallest $r$ which can provide an optimal simulation for G-POD-ROM. The statement "under-resolved reduced-order modeling" corresponds to a numerical simulation where the lengthscale in ROM(e.g., the dimension $r$) is not big or small enough to capture essential information at almost all scales. We first show, in Fig. \ref{3ParasEnergy-nDA8-gammaP-0-gammaU-0-r-17-18-19-20-21}, the temporal evolution of quantities of interest for G-POD-ROM(i.e., $\gamma_u=\gamma_p=0$) and we can see from this figure that $r=19$ could be considered as the optimal value. We, therefore, show in Fig. \ref{3ParasEnergy-nDA8-gammaU-0-100-gammaP-0-100-r8} where $r=8$ that under-resolved reduced-order modeling occurs which manifests itself as the gradually deviating from DNS's in terms of quantities of interest. At that point, we find that the emergence of CDA pulls the temporal evolution of those aforementioned quantities back to the correct tracks, which means that CDA could play a role of stabilization to provide stable and high-accuracy solutions. This is analogous to solving the convection-dominated problem by FEM, where one way is to refine the lengthscale $h$ to serve as DNS, which inevitably leads to huge computational costs; while the other way which is significantly computationally cheap is to utilize the stabilization technique with not too small $h$.
    \begin{figure}
    	\centering
    	\includegraphics[width=1.1\linewidth]{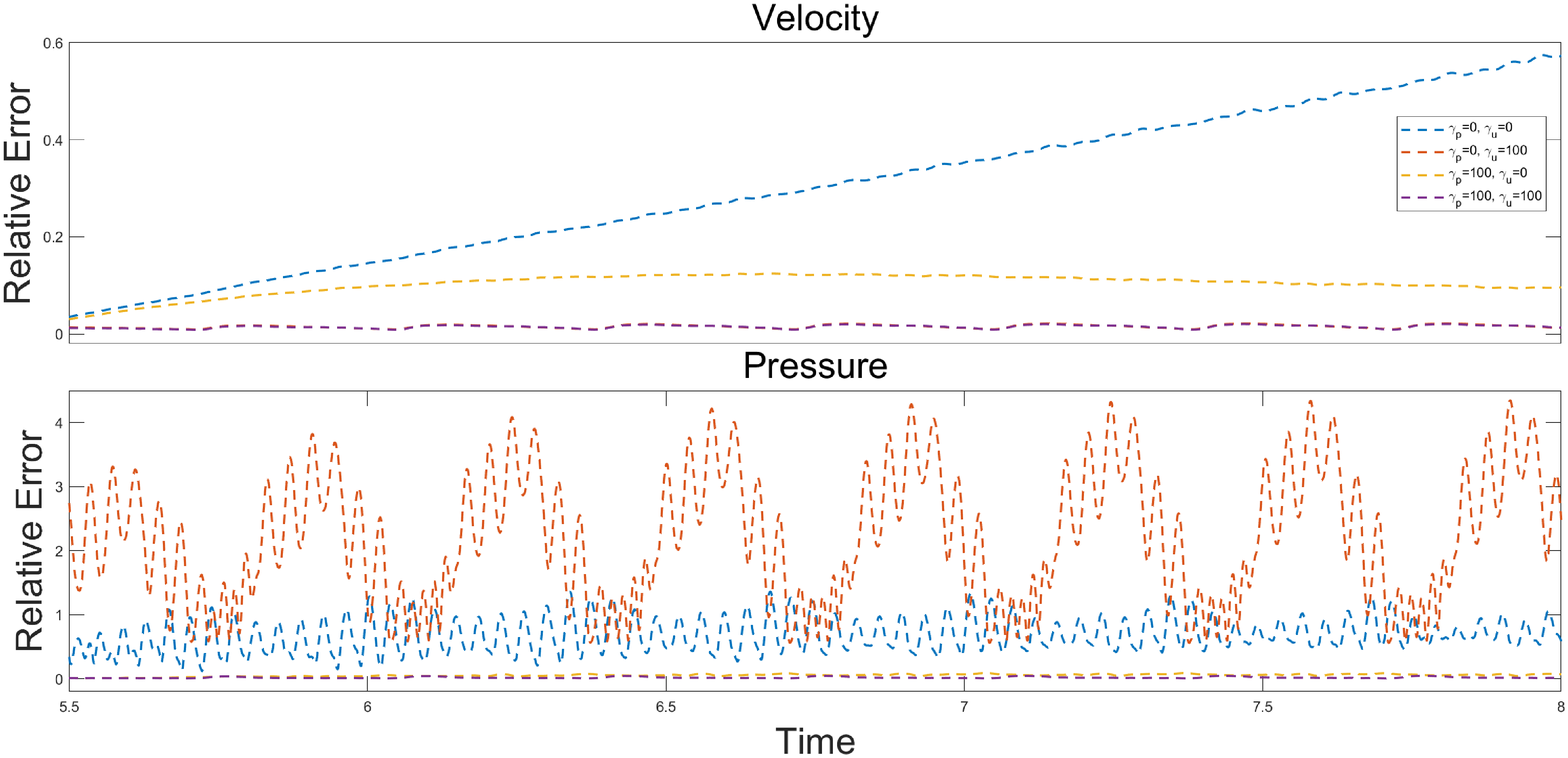}
    	\caption{Relative error's temporal evolution of velocity and pressure using $r=19$ modes with $\gamma_u=0,\,100$, $\gamma_p=0,\,100$.}\label{RelaErr-nDA8-GammaP_0_100-GammaU_0_100-r-19}
    \end{figure}
    \begin{figure}
    	\centering
    	\includegraphics[width=1.1\linewidth]{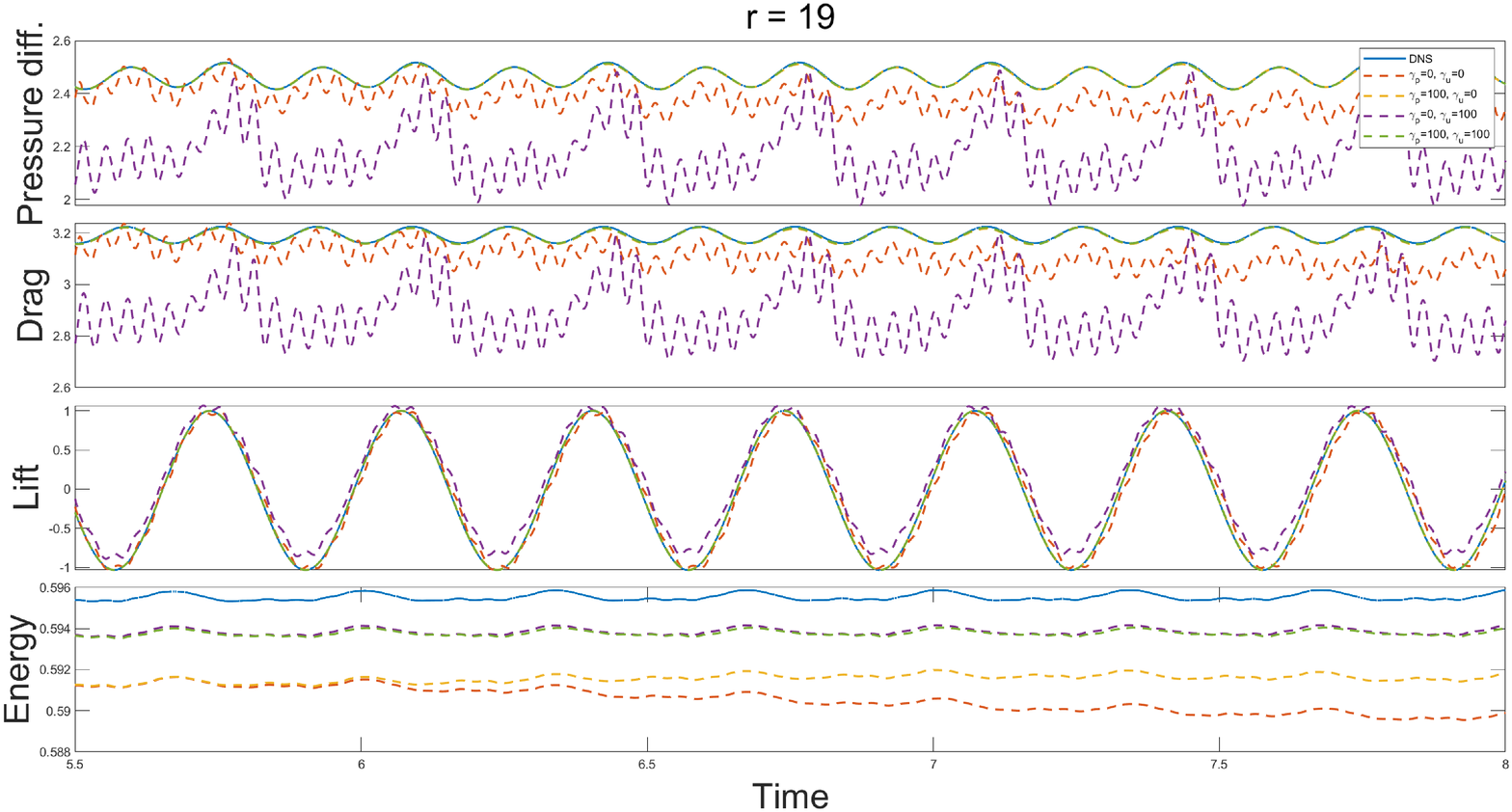}
    	\caption{Temporal evolution of pressure difference at final time, lift and drag coefficients, and kinetic energy using $r=19$ modes with $\gamma_u=0,\,100$, $\gamma_p=0,\,100$.}\label{3ParasEnergy-nDA8-gammaU-0-100-gammaP-0-100-r19}
    \end{figure}
    \begin{figure}
    	\centering
    	\includegraphics[width=1.1\linewidth]{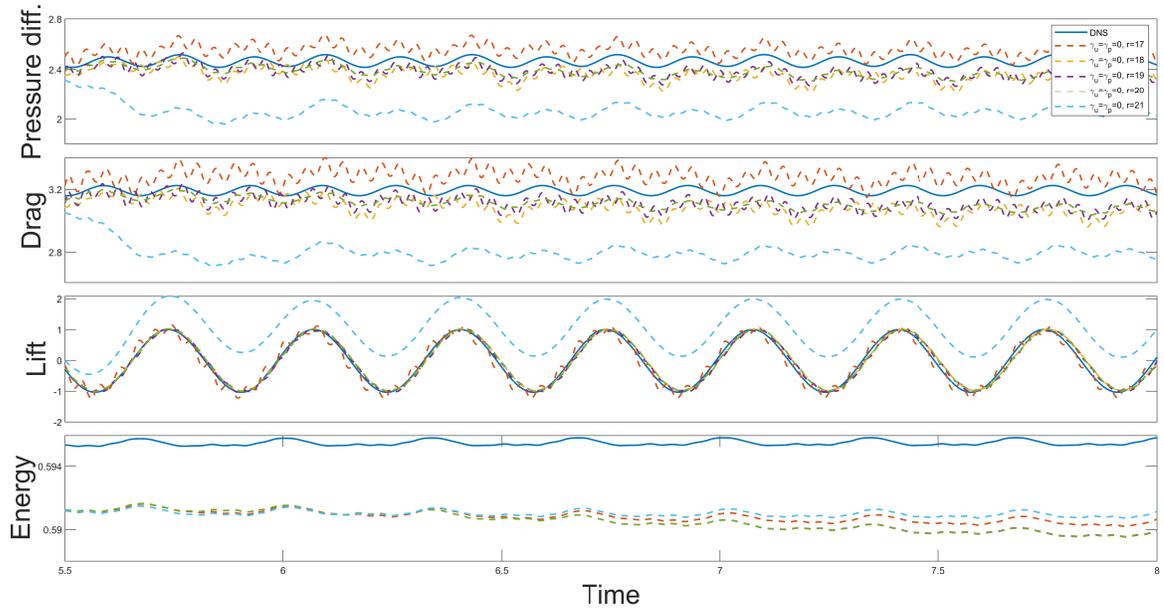}
    	\caption{Temporal evolution of pressure difference at final time, lift and drag coefficients, and kinetic energy with difference POD modes at $\gamma_u=\gamma_p=0$.}\label{3ParasEnergy-nDA8-gammaP-0-gammaU-0-r-17-18-19-20-21}
    \end{figure}
    \begin{figure}
    	\centering
    	\includegraphics[width=1.1\linewidth]{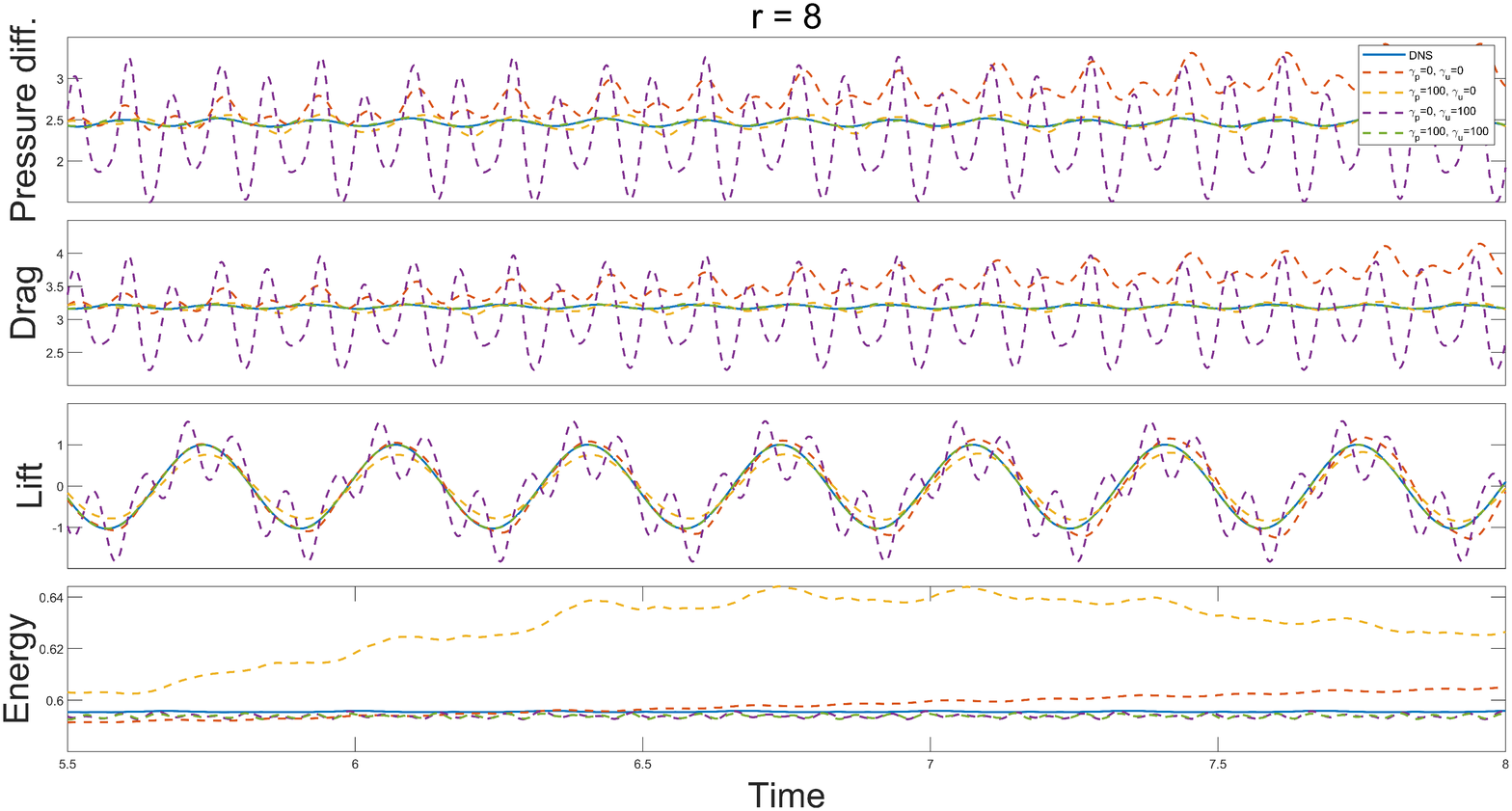}
    	\caption{Temporal evolution of pressure difference at final time, lift and drag coefficients, and kinetic energy using $r=8$ modes with $\gamma_u=0,\,100$, $\gamma_p=0,\,100$.}\label{3ParasEnergy-nDA8-gammaU-0-100-gammaP-0-100-r8}
    \end{figure}	
	\section{Conclusion}\label{section-5}
	In this manuscript, we have proposed a novel, alternative technique to stabilize the reduced-order pressure, and this technique permits ROM to compute the reduced-order pressure \textit{a prior} that is needed in the computation of some relevant quantities of interest, such as drag and lift coefficients and pressure difference on bodies in flow past a cylinder scenarios. With respect to other existing approaches that also can obtain stable and high accurate reduced-order pressure, this ROM is mainly superior on two aspects: \\  
	\indent On the one hand, it permits to obtain stable reduced-order pressure \textit{a prior}, thus avoiding the verification of discrete inf-sup condition over mixed POD velocity-pressure spaces and this compatibility’s validation still keeps an open questions, rather than \textit{a posterior} which would introduce artificial boundary condition if choosing pressure Poisson equation technique or would face the choice of stabilized parameters if choosing pressure-stabilized techniques, like SUPG or PSPG. \\  
	\indent On the other hand, the simple and modular implementation of the classical projection method, followed by the single-mesh’s operation after adopting the algebraic nudging in \cite{CDA-2021-NMPDE}, makes the new ROM more efficient compared to other coupled or residual-based stabilization ROM, which further manifests the strength of ROM over FOM. Numerical analysis gives theoretical guarantees and numerical experiments shows the advantages of our ROM in terms of pressure stability and high accuracy.  \\
	\indent We plan to extend the current work from two aspects: one is to extend NSE to more practical models/equations, like Rayleigh–Bénard convection, such as \cite{RBC-POD-ROM-2019-AMM,RBC-POD-ROM-2020-ICHMT}, thereby broadening application scenes about this CDA-based ROM; another idea deserving carefully consideration is sensitivity analysis about parameters(i.e. viscosity parameter in NSE) between the FOM and ROM, similar works include sensitivity analysis applying to CDA in \cite{CDA-SensAnal-2021-JNS}, parameter varying between FOM and ROM in \cite{ParaVary-2022-CMAME} and parameter recovery in \cite{ParaRecov-2020-SISC}.
	
	\appendix		
	\section{Proof of Lemma \ref{Lemma-Stab}} \label{Proof-Stab}	
	\begin{proof}
		The techniques used here are inspired by \cite{POD-LPS-SINUM-2021}. We firstly assume
		\begin{equation}\label{dt-h2}
			\Delta t \leq Ch^2,
		\end{equation}
		then, 
		\begin{equation}\label{Cons-unh-Linfty}
		\begin{aligned}		
			\|\boldsymbol{u}^n_h\|_{\infty} &\leq \|\boldsymbol{u}^n_h - \boldsymbol{s}^n_h\|_{\infty} + \|\boldsymbol{s}^n_h\|_{\infty} \\
			&\leq Ch^{-1}\|\boldsymbol{u}^n_h - \boldsymbol{s}^n_h\| + C_{1,S}\left(\|\boldsymbol{u}^n\|\|\boldsymbol{u}^n\|_{2}\right)^{1/2} \\
			&\leq Ch^{-1}\left(\|\boldsymbol{u}^n_h - \boldsymbol{u}^n\| + \|\boldsymbol{u}^n - \boldsymbol{s}^n_h\|\right) + C_{1,S}\left(\|\boldsymbol{u}^n\|\|\boldsymbol{u}^n\|_{2}\right)^{1/2}	\\
			&\leq Ch^{-1}\left(h^2 + \Delta t + h^2\right) + C_{1,S}\left(\|\boldsymbol{u}^n\|\|\boldsymbol{u}^n\|_{2}\right)^{1/2} \\
			&\leq Ch + C_{1,S}\left(\|\boldsymbol{u}^n\|\|\boldsymbol{u}^n\|_{2}\right)^{1/2} \\
			&:= C_{L^{\infty}}.
		\end{aligned}
	    \end{equation}		
		Similarly, 
		\begin{equation}\label{Cons-unh-W1infty}
		\begin{aligned}		
			\|\nabla \boldsymbol{u}^n_h\|_{\infty} &\leq \|\nabla (\boldsymbol{u}^n_h - \boldsymbol{s}^n_h)\|_{\infty} + \|\nabla \boldsymbol{s}^n_h\|_{\infty} \\
			&\leq Ch^{-1}\|\nabla (\boldsymbol{u}^n_h - \boldsymbol{s}^n_h)\| + C\|\nabla \boldsymbol{u}^n\|_{\infty} \\
			&\leq Ch^{-1}\left(\|\nabla (\boldsymbol{u}^n_h - \boldsymbol{u}^n)\| + \|\nabla (\boldsymbol{u}^n - \boldsymbol{s}^n_h)\|\right) + C\|\nabla \boldsymbol{u}^n\|_{\infty}	\\
			&\leq Ch^{-1}h^{-1}\left(h^2 + \Delta t\right) + C\|\nabla \boldsymbol{u}^n\|_{\infty} \\
			&\leq C + C\|\nabla \boldsymbol{u}^n\|_{\infty} \\
			&:= C_{W^{1,\infty}}.
		\end{aligned}
		\end{equation}
		Moreover, using the definition of POD modes $\boldsymbol{\varphi}_k$, the fact that $\|\boldsymbol{\varphi}_k\|=1$, the descending property of eigenvalues $\lambda_i$, the orthogonality of eigenvector $x^n$ and the construction of POD basis is derived from the eigen-decomposition of correlation matrix, thus, 
		$$
		\begin{aligned}
			\left\| \boldsymbol{u}^n_h - \Pi^u_r\boldsymbol{u}^n_h \right\| &= \left\|\sum_{i=1}^{d_u}(\boldsymbol{u}^n_h, \varphi_i)\varphi_i - \sum_{i=1}^{r_u}(\boldsymbol{u}^n_h, \varphi_i)\varphi_i\right\| = \left\|\sum_{i=r_u+1}^{d_u}(\boldsymbol{u}^n_h, \varphi_i)\varphi_i \right\|  \\
			&= \left\|\sum_{i=r_u+1}^{d_u}\sum_{j=1}^{M}\frac{1}{\sqrt{M}}\frac{1}{\sqrt{\lambda_i}}x^j_i(\boldsymbol{u}^n_h, u^j_h)\varphi_i \right\|  \\
			&= \left\|\sum_{i=r_u+1}^{d_u}\frac{1}{\sqrt{M}}\frac{1}{\sqrt{\lambda_i}}\left(\sum_{j=1}^{M}(\boldsymbol{u}^n_h, \boldsymbol{u}^j_h)x^j_i\right)\varphi_i \right\|  = \left\|\sum_{i=r_u+1}^{d_u}\frac{1}{\sqrt{M}}\frac{1}{\sqrt{\lambda_i}}M\lambda_ix^n_i\varphi_i \right\|  \\
			&= \sqrt{M}\left( \sum_{i=r_u+1}^{d_u}\lambda_i|x^n_i|^2 \right)^{1/2} \leq \sqrt{M}\sqrt{\lambda_{r_u+1}}\left( \sum_{i=r_u+1}^{d_u}|x^n_i|^2 \right)^{1/2} \leq \sqrt{M}\sqrt{\lambda_{r_u+1}}.
		\end{aligned}
	    $$
		Therefore,	
		\begin{equation}\label{Cons-Stab-0}
		\begin{aligned}
			\|\Pi^u_r \boldsymbol{u}^n_h\|_{\infty} &\leq \|\boldsymbol{u}^n_h\|_{\infty} + \|\boldsymbol{u}^n_h - \Pi^u_r \boldsymbol{u}^n_h\|_{\infty} \\
			&\leq C_{L^{\infty}} + Ch^{-1}\|\boldsymbol{u}^n_h - \Pi^u_r \boldsymbol{u}^n_h\| \\
			&\leq C_{L^{\infty}} + Ch^{-1}\sqrt{M}\sqrt{\lambda_{r_u+1}} \\
			&:= C_{stab,0}.
	    \end{aligned}
        \end{equation}
		Similarly, using Assumption \ref{Assu} and inverse inequality (\ref{InveIneq}), we get
		\begin{equation}
		\begin{aligned}\label{Cons-Stab-1}
			\|\nabla\Pi^u_r \boldsymbol{u}^n_h\|_{\infty} &\leq \|\nabla \boldsymbol{u}^n_h\|_{\infty} + \|\nabla(\boldsymbol{u}^n_h - \Pi^u_r \boldsymbol{u}^n_h)\|_{\infty} \\
			&\leq C_{W^{0,\infty}} + Ch^{-1}\|\nabla(u^n_h - \Pi^u_r u^n_h)\| \\
			&\leq C_{W^{0,\infty}} + Ch^{-1}\sqrt{\sum^{d_u}_{i=r_u+1}\lambda_{i}\|\nabla\varphi_i\|^2} \\
			&:= C_{stab,1}.
	    \end{aligned}
       \end{equation}	
		\indent Finally, at the end of this proof, we remark that the factor $h^{-1}$ appearing in both upper-bound constant $C_{stab,0},C_{stab,1}$ seems to have an negative impact on the stability of these two constants; nevertheless, we notice that the other factor $\lambda_i$, $i\geq r_u+1$ can slow down the trend that $h^{-1}$ becomes too large. In computation, in order to guarantee a good balance between computational complexity and ROM accuracy, $h$ will not be too small and also $\lambda_{r_u+1}$ not too big, which keeps the magnitude $\mathcal{O}(h^{-1}\sqrt{\lambda_{r_u+1}})$ within a reasonable range. For example, in previous works, $h^{-1}\sqrt{\lambda_{r_u+1}} \approx \mathcal{O}(10^{-1})$ in \cite{POD-LPS-SINUM-2020,POD-LPS-SINUM-2021} for benchmark tests, and also $h^{-1}\sqrt{\lambda_{r_u+1}} \approx \mathcal{O}(10^{-1})$ in \cite{Luo-RO-Extra-2019-JMAA,MyFirstPaper} for known smooth solutions. 		
	\end{proof}	
			
	\bibliographystyle{plain}
	\bibliography{Reference}	
\end{document}